\documentclass{amsart}
\usepackage{amssymb}
\usepackage{amsmath}
\usepackage{amsthm}
\usepackage[normalem]{ulem}
\usepackage{verbatim}

\usepackage[utf8]{inputenc}
\usepackage[T1]{fontenc}
\usepackage{selinput}
\SelectInputMappings{%
	adieresis={ä},
	eacute={é},
	Lcaron={Ľ},
}

\newtheorem{theorem}{Theorem}[section]
\newtheorem{lemma}[theorem]{Lemma}

\theoremstyle{definition}

\newtheorem{corollary}[theorem]{Corollary}

\theoremstyle{remark}
\newtheorem{remark}[theorem]{Remark}

\numberwithin{equation}{section}

\newcommand{\Lie}{\mathrm{Lie}}

\newcommand{\Hom}{\mathrm{Hom}}

\newcommand{\Gal}{\mathrm{Gal}}
\newcommand{\ind}{\mathrm{ind}}
\newcommand{\Res}{\mathrm{Res}}
\newcommand{\loc}{\mathrm{loc}}
\newcommand{\fin}{\mathrm{fin}}

\newcommand{\mult}{\mathrm{mult}}

\newcommand{\Q}{\mathbb Q}

\newcommand{\Z}{\mathbb Z}
\newcommand{\N}{\mathbb N}

\begin{document}

\title[The pro-$p$ Iwahori subgroup of $GL(n)$ and base change]{Globally analytic principal series representation \\ and Langlands base change}


%    author one information
\author{Jishnu Ray}
\address{Département de Mathématiques, Bâtiment 307,
Faculté des Sciences d'Orsay, Université Paris-Sud,
F-91405 Orsay Cedex}
\curraddr{}
\email{jishnuray1992@gmail.com; jishnu.ray@u-psud.fr}
\thanks{}

%    author two information
%\author{}
%\address{}
%\curraddr{}
%\email{}
%\thanks{}

%    \subjclass is required.
%\subjclass[2010]{Primary }

%\date{\today}

%\dedicatory{}

%    Abstract is required.
%\begin{abstract}
%Here is an abstract.
%\end{abstract}

\maketitle

\section{Introduction}\label{sub:intro}
The paper \cite{Clozel1}  deals with the construction of base change of globally analytic distributions on the pro-$p$ Iwahori groups (seen as rigid-analytic spaces) which is compatible with the $p$-adic Langlands correspondence in the case of principal series of $GL(2)$, and this is what we will extend to $GL(n)$ in this article. Of course, we need to take $p>n+1$, so that the pro-$p$ Iwahori subgroup of $GL(n)$ is $p$-saturated in the sense of Lazard \cite[III, $3.2.7.5$]{Lazard} and isomorphic analytically to the product $\mathbb{Z}_p^d$ \cite[III, $3.3.2$]{Lazard} (for some $d$, here $\mathbb{Z}_p$ is seen as a rigid-analytic closed ball of radius $1$).  

In section \ref{sub:obtaing}, for a finite unramified extension $L$ over $\mathbb{Q}_p$, we briefly recall the notion of restriction of scalars functor in the context of rigid-analytic spaces. In section \ref{sub:holomorphicbasechange}, we give the basic definitions of holomorphic and Langlands base change maps
following \cite{Clozel1}.  Section \ref{sub:sectionbigger} treats the case of principal series of $GL(n)$. Specifically, denote by $G$ the pro-$p$ Iwahori subgroup of $GL_n(\mathbb{Z}_p)$ (the group of matrices in $GL_n(\mathbb{Z}_p)$ that are lower unipotent modulo $p\mathbb{Z}_p$), $B$ the subgroup of matrices in $GL_n(\mathbb{Z}_p)$ which are lower triangular modulo $p\mathbb{Z}_p$, $P_0 \supset T_0$  the set of upper triangular (resp. diagonal) matrices in $B$, $Q_0=P_0 \cap G$, $P^+$ the Borel subgroup of upper triangular matrices in $GL_n(\mathbb{Z}_p)$, $W$ the Weyl group (isomorphic to $S_n$) of $GL_n(\mathbb{Q}_p)$ with respect to its maximal torus, $P_w^+=B \cap wP^+w^{-1},w \in W,$   ${\chi: T_0 \rightarrow K^{\times}}$  a locally analytic character with $\chi(t_1,...,t_n)=\chi_1(t_1)\cdots \chi_n(t_n),$
and $\chi_i(t)=t^{c_i}$, where $c_i=\frac{d}{dt}\chi_i(t)|_{t=1}$, for $t$ sufficiently close to $1$,  $I_{\loc}$ be the locally analytic functions
\begin{equation*}
I_{\loc}=\{f \in \mathcal{A}_{\loc}(G,K):f(gb)= \chi(b^{-1})f(g),b \in Q_0,g \in G\},
\end{equation*}
 with $\mathcal{A}_{\loc}(G,K)$  the locally analytic functions having values in an unramified extension $K$ over $\mathbb{Q}_p$. Note that the vector space of locally analytic principal series  $\ind_{P_0}^B(\chi)_{\loc}$ 
 is isomorphic to the space  $I_{\loc}$, which are locally analytic functions from $\Z_p^d \rightarrow K$ for  an appropriate dimension $d$ (cf. section \ref{sub:sectionbigger}). "Locally analytic functions" mean that locally around a neighborhood of a point, the functions  can be written in the form of power series with coefficients in $K$. 
 
 Then, we show in lemmas \ref{lem:lemmadiagonal1}, \ref{lem:lemmalower1}, \ref{lem:lemmaupper2} that the action of  $G$ on the \textit{globally analytic vectors} of $I_{\loc}$ given by $h\cdot f(g) \longmapsto f(h^{-1}g), \text{          } (h \in G)$ is a globally analytic action in the sense of Emerton \cite{Emertonbook}. Here the globally analytic functions of $I_{\loc}$ are the Tate algebra of functions from $\Z_p^d \rightarrow K$ which can be written as power series on the affinoid rigid-analytic space $\Z_p^d$ with coefficients in $K$ going to $0$ (section \ref{sub:sectionbigger}) . For a detailed discussion on globally analytic representation see \cite{Emertonbook}.
 
 The requisite condition of analyticity of $\chi$ is treated in \ref{eq:unramifiedK}. Let $\mu$ be the linear form from the Lie algebra of the torus $T_0$ to $K$ given by 
  $$\mu=(-c_1,...,-c_n):Diag(t_1,...,t_n)\mapsto \sum_{i=1}^n-c_it_i$$
 where $t=(t_i) \in \Lie(T_0)$. For negative root $\alpha =(i,j), i>j$, let $H_{(i,j)}=E_{i,i}-E_{j,j}$ where $E_{i,i}$ is the standard elementary matrix.

 Then,  we show that (see theorem \ref{thm:holomorphic} and theorem \ref{thmmainglo}):\\

\textit{\textbf{Theorem.}} \textit{Assume } $p>n+1$ and $\chi$ is analytic. \textit{Then the  space of globally analytic vectors of }$\ind_{P_0}^B(\chi)_{\loc}$ \textit{is an } \textit{admissible} \textit{and globally analytic representation of} $G$. \textit{Furthermore the space of globally analytic vectors of } $\ind_{P_0}^B(\chi)_{\loc}$ \textit{ is irreducible if and only if for all } $\alpha=(i,j) \in \Phi^-$,  $-\mu(H_{\alpha}) + i-j \notin \{1,2,3,...\}$.\\

Here, the admissibility is in the sense of \cite{Emertonbook} (see also \cite[sec. $2.3$]{Clozel1}). For global analyticity, we compute explicitly the action of $G$ on the Tate algebra of globally analytic functions of $\ind_{P_0}^B(\chi)_{\loc}$ and show that the action map is a globally analytic function on $G$ seen as a rigid-analytic space. For the irreducibility we first use the action of the Lie algebra of $G$ to show that any non-zero closed $G$-invariant subspace of the globally analytic vectors of  $\ind_{P_0}^B(\chi)_{\loc}$ contains the constant function $1$. The remaining part of the argument for the proof of irreducibility uses the notion of Verma modules and its condition of irreducibility, a result of Bernstein-Gelfand.

 Finally, in section
\ref{sub:lastsectionbasechange}, we extend these results to the pro-$p$ Iwahori group of $GL_n(L)$ where $L$ is an unramified finite extension of $\mathbb{Q}_p$. Then, in theorem \ref{thmlastiwahoribase}, we use the  Steinberg tensor product \cite{Steinberg1} to construct base change in the context of Langlands functoriality. 

In  appendix \ref{appendixA} we deal with the   globally analytic vectors induced from the Weyl orbit of the upper triangular Borel subgroup of the Iwahori subgroup $B$, i.e. the globally analytic vectors of $\ind_{P_w^+}^B(\chi^w)_{\loc}$, where $\chi_w(h)=\chi(w^{-1}hw)$.
\section{Base change maps for analytic functions} 
We introduce the basic notions of rigid-analytic geometry including a brief discussion on the restriction of scalars. Then  we briefly recall (following \cite{Clozel1}) the notions of holomorphic and Langlands base change functors producing from a globally analytic representation over $\Q_p$, to a representation over $L$. The Langlands base change is related to the "Steinberg tensor product" described at the end of section $1.1$ of \cite{Clozel2} for $GL(2)$. 
\subsection{}\label{sub:obtaing}
Let $L$ be a finite unramified extension of $\mathbb{Q}_p$ of degree $N$, $(B^1/L)$ be the (rigid-analytic) closed unit ball over $L$ with its Tate algebra of analytic functions $\mathcal{T}_L=L\langle x \rangle$, $G_L$ be a rigid-analytic group isomorphic as a rigid analytic space to  $(B^1/L)^d$ which is a rigid-analytic space with affinoid algebra ${\mathcal{A}(G_L):=\widehat{\otimes}^d\mathcal{T}_L=L\langle x_1,...,x_d \rangle}$, the Tate algebra of analytic functions in $d$ variables with coefficients in $L$. (With the notations of section \ref{sub:intro}, for $L=\mathbb{Q}_p$, we can take $G_L$ to be the pro-$p$ Iwahori group $G$ assuming $p>n+1$). The restriction of scalars functor \cite{bertapelle} associates to $G_L$ a rigid analytic space $\Res_{L/\mathbb{Q}_p}G_L$ over $\mathbb{Q}_p$. In general, this functor does not behave trivially, but $L$ being unramified, we obtain 
\begin{equation*}
\Res_{L/\mathbb{Q}_p}(B^1/L)\cong (B^1/\mathbb{Q}_p)^N,
\end{equation*}
\cite[lemma $1.1$]{Clozel1} which is canonically obtained by the choice of a basis $(e_i)$ of $\mathcal{O}_L$ over $\mathbb{Z}_p$. Precisely, for an affinoid $\mathbb{Q}_p$-algebra $\mathcal{B}$ and for   $f \in \Hom_L(L\langle x \rangle, \mathcal{B} \otimes_{\mathbb{Q}_p}L)$ with $f(x)=\sum b_ie_i \text{  } (  b_i \in \mathcal{B})$, we canonically define a function $g \in \Hom_{\mathbb{Q}_p}(\mathbb{Q}_p \langle x_1,...,x_N \rangle, \mathcal{B})$ with $g(x_i)=b_i$ which is given by
 
 \begin{equation*}
 g(x_1,...,x_N)=f(\sum e_ix_i),
 \end{equation*}
   \cite[section $1.1$]{Clozel1}. As the restriction of scalars is compatible with direct products \cite[prop. $1.8$]{bertapelle}, $\Res_{L/\mathbb{Q}_p}G_L = (B^1/\mathbb{Q}_p)^{dN}.$ Henceforth, we write $\Res \text{ } G_L$ to denote $\Res_{L/\mathbb{Q}_p}G_L$.

\subsection{}\label{sub:holomorphicbasechange}
   Assume now that $G_L\cong (B^1/L)^d $ is obtained by \textit{extension} of scalars from $\mathbb{Q}_p$. Then, the Tate algebra $\mathcal{A}(G_L)=\mathcal{A}(G_{\mathbb{Q}_p}) \otimes L.$  The co-multiplication map $m^*$, defined by a morphism 
\begin{equation*}
m^*:\mathcal{A}(G_L) \rightarrow \mathcal{A}(G_L) \widehat{\otimes} \mathcal{A}(G_L)
\end{equation*}
with image inside the completed tensor product,  is obtained by extension of scalars from 
\begin{equation*}
m_0^*:\mathcal{A}(G_{\mathbb{Q}_p}) \rightarrow \mathcal{A}(G_{\mathbb{Q}_p}) \widehat{\otimes} \mathcal{A}(G_{\mathbb{Q}_p}).
\end{equation*}
To an  analytic function $f \in \mathcal{A}(G_L)$, we associate a function ${g \in \mathcal{A}(\Res\text{ } G_L) \otimes L}$, $\Res\text{ } G_L$ defined as in section \ref{sub:obtaing}. Then, by composing with the natural map $\mathcal{A}(G_{\mathbb{Q}_p}) \rightarrow \mathcal{A}(G_L)$, we obtain a "holomorphic base change" map 
\begin{equation*}
b_1:\mathcal{A}(G_{\mathbb{Q}_p}) \rightarrow \mathcal{A}(\Res \text{ }G_L) \otimes L.
\end{equation*}
The Galois group $\Sigma=\Gal(L/\mathbb{Q}_p)$ of the unramified Galois extension $L$ acts naturally on $G_L$ (by automorphisms on the Tate algebra) and acts on $\Res \text{ }G_L$ by $\mathbb{Q}_p$-automorphism. Define the map
\begin{align*}
b:\mathcal{A}(G_{\mathbb{Q}_p}) &\rightarrow \mathcal{A}(\Res \text{ }G_L) \otimes L\\
b(f) &= \prod_{\sigma \in \Sigma}b_1(f)^{\sigma}.
\end{align*}
Then, by \cite[prop. $1.5$]{Clozel1}, the natural maps $b_1,b$ commute with co-multiplications and under the isomorphism $\mathcal{A}_L(\Res \text{ }G_L) \cong \widehat{\otimes}_{\sigma} \text{ } \mathcal{A}(G_L)$, the map $b=\otimes_{\sigma} \text{ }b_1^{\sigma}$ (the isomorphism $\mathcal{A}_L(\Res \text{ }G_L) \cong \widehat{\otimes}_{\sigma} \text{ } \mathcal{A}(G_L)$ follows from  $\Res \text{ } G_L \otimes_{\mathbb{Q}_p} L \cong \prod_{\sigma}G_L$, see the discussion before proposition $1.5$ of \cite{Clozel1}). 

Fix  a finite extension $K$ of $\mathbb{Q}_p$ and an injection $i: L \subset K$. If $\sigma \in \Gal(L/\mathbb{Q}_p)$, we then have the injection $i\circ \sigma: L \rightarrow K$. Denote by $V$ a (globally) analytic representation of $G_{\mathbb{Q}_p}$ on a $K$- Banach space. Then $V$ naturally extends to an analytic representation of $G_L$; this is called the \textit{holomorphic base change} of $V$ in \cite{Clozel1}. For $\sigma \in \Gal(L/\mathbb{Q}_p)$, write $V^{\sigma}$ the representation of $G_L$ associated to $i\circ \sigma$. Then, the \textit{full (Langlands) base change} of $V$ is defined to be the globally analytic representation of $\Res_{L/\mathbb{Q}_p}(G_L)$ on $\widehat{\otimes}_{\sigma} \text{ } V^{\sigma}$ (cf. \cite[def. $3.2$]{Clozel1}). 
\section{Globally analytic principal series for $GL(n)$}\label{sec:twoiwahori}
We first recall the notion of locally analytic principal series representation induced from the Borel to the Iwahori subgroup of $GL(n,\Z_p)$. Then we treat the action of the pro-$p$ Iwahori on the subspace of rigid-analytic functions within the locally analytic principal series and show that this action is a globally analytic action (theorem \ref{thmmainglo}). This gives us the globally analytic induced principal series representation under the pro-$p$ Iwahori subgroup $G$. Furthermore, we treat the condition of irreducibility of the globally analytic principal series by translating an irreducibility condition of a suitable Verma module (theorem  \ref{thm:holomorphic}). Finally in section \ref{sub:lastsectionbasechange} we base change our globally analytic representation to a finite unramified extension $L$ of $\Q_p$.
\subsection{~}\label{sub:sectionbigger}
 We consider the case of principal series for $GL_n(\mathbb{Z}_p)$. Denote by $G$ the pro-$p$ Iwahori subgroup of $GL_n(\mathbb{Z}_p)$, i.e. the group of matrices in $GL_n(\mathbb{Z}_p)$ that are lower unipotent modulo $p\mathbb{Z}_p$, $B$ the  subgroup of matrices in $GL_n(\mathbb{Z}_p)$ which are lower triangular modulo $p\mathbb{Z}_p$, $P_0 \supset T_0$ be the set of upper triangular (resp. diagonal) matrices in $B$, $\chi: T_0 \rightarrow K^{\times}$ be a locally analytic character with 
\begin{equation*}
\chi(t_1,...,t_n)=\chi_1(t_1)\cdots \chi_n(t_n),
\end{equation*}
and $\chi_i(t)=t^{c_i}$ where $c_i=\frac{d}{dt}\chi_i(t)|_{t=1}$ for $t$ sufficiently close to $1$. Hence, $c_i \in K$. 

We first consider, as in \cite{Clozel1}, the locally analytic induced representation of $B$,
\begin{equation*}
J_{\loc}=\ind_{P_0}^B(\chi)_{\loc}=\{f \in \mathcal{A}_{\loc}(B,K):f(gb)= \chi(b^{-1})f(g),b \in P_0, g \in B\},
\end{equation*} 
 where $\chi$ is naturally extended to $P_0$ and $\mathcal{A}_{\loc}(B,K)$ is the space of  locally analytic functions on $B$. With $U$ the lower unipotent subgroup of $B$ with entries in $\mathbb{Z}_p$ in the lower triangular part, $1$ in the diagonal entries and $0$ elsewhere, we have the natural decomposition 
\begin{equation}\label{eq:iwaho}
B=UP_0
\end{equation}
Since $\chi$ is fixed, the restriction of the functions of $J_{\loc}$ to $G \subset B$ is injective. With $Q_0=P_0 \cap G$, we deduce that the vector space of $J_{\loc}$ is 
\begin{equation}\label{eq:chiaction}
I_{\loc}=\{f \in \mathcal{A}_{\loc}(G,K):f(gb)\equiv \chi(b^{-1})f(g),b \in Q_0, g \in G\}.
\end{equation}
 With the decomposition $G=UQ_0$, we see that $I_{\loc}\cong \mathcal{A}_{\loc}(\mathbb{Z}_p^{\frac{n(n-1)}{2}},K)=\mathcal{A}_{\loc}(U,K)$. Here, $\mathbb{Z}_p$ is seen as the rigid analytic (additive) group $B^1(\mathbb{Z}_p)$. The group $G$ acts by left translation
\begin{equation}\label{eq:Gaction}
h\cdot f(g) \longmapsto f(h^{-1}g).
\end{equation}

Let $E_{i,j}$ be the elementary matrices with $1$ in the $(i,j)^{th}$ place and $0$ elsewhere.
From now on, we assume 
\begin{equation}\label{eq:conditiononp}
p > n+1,
\end{equation}
then $G$ is $p$-saturated in the sense of Lazard \cite[III, $3.2.7.5$]{Lazard} and thus, it is the ordered product (as a rigid analytic group) of the following one-parameter subgroups:
 \begin{enumerate}
 \item first, for $y \in \mathbb{Z}_p,$ take the one-parameter lower unipotent matrices by the following lexicographic order: the $1$-parameter group of matrices $(1+yE_{i,j})$ comes before the $1$-parameter group of matrices$(1+yE_{k,l})$ if and only if $i<k$ or $i=k$ and $ j<l$, \label{lexi}\\
 
 \item  then, for $t_k \equiv 1[p]$ and $k \in [1,n]$, take the one-parameter diagonal subgroups $(t_kE_{k,k}+\sum_{i=1,i \neq k}^nE_{i,i})$ starting from the top left extreme to the low right extreme and, \\
 
 \item finally, for $y \in p\mathbb{Z}_p$, take the upper unipotent matrices in the following order: the $1$-parameter group of matrices $(1+yE_{i,j})$ comes before the $1$-parameter group of matrices $(1+yE_{k,l})$ if and only if $i\geq k$ or $i=k$ and  $j>l$. 
 \end{enumerate}
 That is, for the lower unipotent matrices, we start with the top and left extreme and then fill the lines from the left, going down and for the upper unipotent matrices we start with the low and right extreme and then fill the lines from the right, going up. (Cf. \cite[III, $3.3.2$]{Lazard} for the rigid-analyticity and see theorem $2.2.1$ and remark $2.2.2$ of \cite{Ray2} for the order of the product i.e. an ordered Lazard basis of $G$, although in \cite{Ray2} we have taken $G$ to be upper unipotent matrices modulo $p$ but this does not matter).
 
 Let now $\mathcal{A}=\mathcal{A}(U,K)=\mathcal{A}(\mathbb{Z}_p^{\frac{n(n-1)}{2}},K)$ be the subspace of globally analytic functions of $I_{\loc}=\mathcal{A}_{\loc}(U,K)$. Thus  $f \in \mathcal{A}$ is a globally analytic function in the variables $a_{i,j}$ on $U$, that is, 
 $$f(A)=\sum_{\nu \in \mathbb{N}^d}c_{\nu}a^{\nu}$$
such that $c_{\nu} \in K$ and $|c_{\nu}| \rightarrow 0$ as $|\nu| \rightarrow \infty$. 
 Here ${d=\frac{n(n-1)}{2}}$,  $a=(a_{2,1},a_{3,1},a_{3,2},...,a_{n,n-1}) \in \mathbb{Z}_p^d$ with the lexicographic ordering of $a_{i,j}$ as in (\ref{lexi}),  $\nu=(\nu_{2,1},\nu_{3,1},...,\nu_{n,n-1}) \in \mathbb{N}^d, a^{\nu}=a_{2,1}^{\nu_{2,1}}\cdots a_{n,n-1}^{\nu_{n,n-1}}$ and $|\nu|=\nu_{2,1}+\cdots + \nu_{n,n-1}$.

 We now seek conditions such that if $f$ is a globally analytic function on $G$ and the action of $G$ is defined as above then, the map $h \longmapsto h \cdot f(g)=f(h^{-1}g)$ is globally analytic. 
 \begin{lemma}\label{lem:sufficeaction}
 With the above notations, for $p>n+1$, the action of $G$ on $f \in \mathcal{A}(U,K)$, i.e the map $h\mapsto h \cdot f$ is a globally analytic function on $G$ if and only if it is so for all $1$-parameter (rigid-analytic) subgroups and the diagonal subgroup of which $G$ is the product.
 \end{lemma}
 \begin{proof}
 Follows from the same argument as in the discussion after lemma $3.4$ of \cite{Clozel1}.
 \end{proof}
 Thus, our goal is to verify the analyticity of the action of the diagonal subgroup, the $1$-parameter lower unipotent subgroups and the $1$-parameter upper unipotent subgroups of $G$ which are treated in lemmas $\ref{lem:lemmadiagonal1}, \text{ } \ref{lem:lemmalower1}$ and $\ref{lem:lemmaupper2}$ respectively.
 
 Let $A=(a_{i,j})_{i,j}$ be any matrix in $U$ (i.e. $a_{i,i}=1$ and $a_{i,j}=0$ for $i<j$) and  $T=diag(t_1,...,t_n)=\sum_{k=1}^nt_kE_{k,k}$ be any element in the diagonal $T_0 \cap G$, where $t_k \in 1+p\mathbb{Z}_p$.
 Assume $f \in I_{\loc}$, then the action of $T$ on $f$, given by \ref{eq:Gaction}, is 
 \begin{align*}
 T\cdot f(A)=f\Big(Diag(t_1^{-1},...,t_n^{-1})A\Big)&= f\Big((\sum_{k=1}^nt_k^{-1}E_{k,k})(\sum_{i,j=1}^na_{i,j}E_{i,j})\Big),\\
 &=f(\sum_{j,k=1}^nt_k^{-1}a_{k,j}E_{k,j}),\\
 &=f \Big( (\sum_{k,j=1}^nt_k^{-1}t_ja_{k,j}E_{k,j})(\sum_{j=1}^nt_j^{-1}E_{j,j}) \Big),\\
 &=f(\sum_{k,j=1}^nt_k^{-1}t_ja_{k,j}E_{k,j})\chi(t_1,...,t_n) \text{ } (\text{from } \ref{eq:chiaction}),
 \end{align*}
Interchanging indices $k \rightarrow i$, we obtain 
\begin{align}\label{eq:analyticexpression1}
(\sum_{i=1}^nt_iE_{i,i})\cdot f(\sum_{i,j=1}^na_{i,j}E_{i,j})=f(\sum_{i,j=1}^nt_i^{-1}t_ja_{i,j}E_{i,j})\chi(t_1,...,t_n)
\end{align}
with $a_{i,i}=1$, $a_{i,j}=0$ for $i<j$ and $t_i \equiv 1(\mod p)$.

Taking $f=1$ we see that $\chi(t_1,...,t_n)$ must be an analytic function. By \ref{eq:analyticexpression1}, for fixed $k \in [1,n]$ considering the action of the matrix $(t_kE_{k,k}+\sum_{i=1,i \neq k}^nE_{i,i})$ on $f$ we obtain,
\begin{align}
(t_kE_{k,k}+\sum_{i=1,i \neq k}^nE_{i,i})f(A)&=f(\sum_{\underset{u>v}{u,v \neq k}}^{}a_{u,v}E_{u,v}+a_{k,k}E_{k,k}+\sum_{j=1}^{k-1}t_k^{-1}a_{k,j}E_{k,j}+\sum_{i=k+1}^nt_ka_{i,k}E_{i,k}) \label{eq:diagonal1}\\
& \qquad \times \chi(1,...,t_k,...,1)  \label{eq:analyticexpressionone1} \\
&:=f(\mathcal{C})\chi(1,...,t_k,...,1) \label{eq:diag}
\end{align}
where $\mathcal{C}$ is the matrix $(\sum_{\underset{u>v}{u,v \neq k}}^{}a_{u,v}E_{u,v}+a_{k,k}E_{k,k}+\sum_{j=1}^{k-1}t_k^{-1}a_{k,j}E_{k,j}+\sum_{i=k+1}^nt_ka_{i,k}E_{i,k})$. Assume now that $f$ is globally analytic in the variables $a_{i,j}$ on $U$, that is, 
\begin{equation}\label{eq:fanalytic}
f(A)=\sum_{\nu \in \mathbb{N}^d}c_{\nu}a^{\nu},
\end{equation}
such that $c_{\nu} \in K$ and $|c_{\nu}| \rightarrow 0$.
Then with $t_k=1+p\xi_k$, $\xi_k \in \mathbb{Z}_p$, 
\begin{align}
f(\mathcal{C})&=\sum_{\nu}c_{\nu}(a_{k,k}^{\nu_{k,k}}\prod_{\underset{u>v}{u,v \neq k}}^{}a_{u,v}^{\nu_{u,v}})(\prod_{j=1}^{k-1}(t_k^{-1}a_{k,j})^{\nu_{k,j}})(\prod_{i=k+1}^n(t_ka_{i,k})^{\nu_{i,k}})\label{diago}\\
&=\sum_{\nu}c_{\nu}(a_{k,k}^{\nu_{k,k}}\prod_{\underset{u>v}{u,v \neq k}}^{}a_{u,v}^{\nu_{u,v}})(\prod_{j=1}^{k-1}(1+p\xi_k)^{-\nu_{k,j}}a_{k,j}^{\nu_{k,j}})(\prod_{i=k+1}^n(1+p\xi_k)^{\nu_{i,k}}a_{i,k}^{\nu_{i,k}}).\label{eq:biggest}
\end{align}
Recall that for $|v|<1$, $m \in \mathbb{N}$, we have $(1-v)^{-m}=\sum_{q=0}^{\infty}\binom{m+q-1}{q}v^q.$
Now, inserting the expressions \[(1+p\xi_k)^{-\nu_{k,j}}=\sum_{q_{k,j}=0}^{\infty}{{\nu_{k,j}+q_{k,j}-1}\choose{q_{k,j}}}(-p\xi_k)^{q_{k,j}}\] and $(1+p\xi_k)^{\nu_{i,k}}=\sum_{u_{i,k}=0}^{\nu_{i,k}}{{\nu_{i,k}}\choose{u_{i,k}}}p^{u_{i,k}}\xi_k^{u_{i,k}}$
in equation \ref{eq:biggest} we obtain, with $|q|:=q_{k,1}+\cdots+q_{k,k-1}$, $|u|=u_{k+1,k}+\cdots+u_{n,k}$ and $v_{max}=\prod_{i=k+1}^n\nu_{i,k}$,
\begin{align*}
f(\mathcal{C})&=\sum_{\nu}c_{\nu}(a_{k,k}^{\nu_{k,k}}\prod_{\underset{u>v}{u,v \neq k}}^{}a_{u,v}^{\nu_{u,v}})\Big(\prod_{j=1}^{k-1}(\sum_{q_{k,j}=0}^{\infty}{{\nu_{k,j}+q_{k,j}-1}\choose{q_{k,j}}}(-p\xi_k)^{q_{k,j}}a_{k,j}^{\nu_{k,j}})\Big)\\ 
&\qquad \qquad \qquad \qquad \times \Big(\prod_{i=k+1}^n\sum_{u_{i,k}=0}^{\nu_{i,k}}{{\nu_{i,k}}\choose{u_{i,k}}}p^{u_{i,k}}\xi_k^{u_{i,k}}a_{i,k}^{\nu_{i,k}}\Big)\\
&=\sum_{\nu}c_{\nu}(a_{k,k}^{\nu_{k,k}}\prod_{\underset{u>v}{u,v \neq k}}^{}a_{u,v}^{\nu_{u,v}})\Big(\sum_{N \geq 0}^{\infty}\xi_k^N(\sum_{\underset{}{|q|=N}}\prod_{j=1}^{k-1}{{\nu_{k,j}+q_{k,j}-1}\choose{q_{k,j}}}(-p)^{q_{k,j}}a_{k,j}^{\nu_{k,j}})\Big)\\ 
&\qquad \qquad \qquad \qquad \times \Big(\sum_{M=0}^{v_{max}}\xi_k^M(\sum_{|u|=M}\prod_{i=k+1}^n{{\nu_{i,k}}\choose{u_{i,k}}}p^{u_{i,k}}a_{i,k}^{\nu_{i,k}})\Big).\\
\end{align*}
Let $f_N$ and $g_M$ be defined by 
\begin{align}
f_N&=(\sum_{\underset{}{|q|=N}}\prod_{j=1}^{k-1}{{\nu_{k,j}+q_{k,j}-1}\choose{q_{k,j}}}(-p)^{q_{k,j}}a_{k,j}^{\nu_{k,j}}), \label{eq:fndefinition}\\
g_M&=(\sum_{|u|=M}\prod_{i=k+1}^n{{\nu_{i,k}}\choose{u_{i,k}}}p^{u_{i,k}}a_{i,k}^{\nu_{i,k}})\label{eq:gmdefinition}.
\end{align}
Then,
\begin{align}
f(\mathcal{C})&=\sum_{\nu}c_{\nu}(a_{k,k}^{\nu_{k,k}}\prod_{\underset{u>v}{u,v \neq k}}^{}a_{u,v}^{\nu_{u,v}})(\sum_{N \geq 0}^{\infty}\xi_k^Nf_N)(\sum_{M=0}^{v_{max}}\xi_k^Mg_M)\\
&=\sum_{m \geq 0}^{\infty}\xi_k^m\Big(\sum_{\nu}c_{\nu}(a_{k,k}^{\nu_{k,k}}\prod_{\underset{u>v}{u,v \neq k}}^{}a_{u,v}^{\nu_{u,v}})\sum_{N+M=m}f_Ng_M\Big). \label{eq:lardestseen}
\end{align}
Any element $f \in \mathcal{A}(U,K)$ is of the form $f=\sum_{\nu \in \mathbb{N}^d}c_{\nu}a^{\nu}$ with $\lim_{|\nu| \rightarrow \infty}|c_{\nu}|=0$. The space $\mathcal{A}(U,K)$ is a $K$-Banach space with the sup norm on $f$ defined by \[|f|=\sup |c_{\nu}|\]
(cf. \cite[chapter $2$]{Bosch1}). Recall that for any $K$-Banach space $V$ with norm $|\cdot |$, a representation $\pi$ of $G$ on $V$ is called a globally analytic representation if the map \[g \longmapsto g\cdot v =\pi(g)v\]
is globally analytic on $G$ for all $v \in V$. Thus, in coordinates $(x_1,...,x_l)$ with $l=\dim(G)$:
\[g \cdot v =\sum_kx^kv_k\] where $v_k \in V$ and $|v_k|\rightarrow 0$. Here $k=(k_1,...,k_l)$ and $x^k=x_1^{k_1} \cdots x_l^{k_l}, k_i \in \mathbb{N}$ (cf. \cite{Emertonbook}, \cite[section $2$]{Clozel1}).

Now, with $t_k=1+p\xi_k$, $\xi_k \in \mathbb{Z}_p$, in order to show that the action of the one-parameter diagonal subgroup $t_k(E_{k,k})+\sum_{\underset{i\neq k}{i=1}}^{n}E_{i,i}$ on $ f \in \mathcal{A}(U,K)$ is analytic we have to show that the map 
\begin{align*}
\mathbb{Z}_p & \rightarrow \mathcal{A}(U,K)\\
\xi_k &\longmapsto \Big((1+p\xi_k)E_{k,k}+\sum_{i=1,i \neq k}^nE_{i,i}\Big)f=f(\mathcal{C})\chi(1,..,1+p\xi_k,...,1)
\end{align*}
is a globally analytic map on $\mathbb{Z}_p$.
 The norm of the coefficient of $\xi_k^m$, in equation \ref{eq:lardestseen}, is \[|\Big(\sum_{\nu}c_{\nu}(a_{k,k}^{\nu_{k,k}}\prod_{\underset{u>v}{u,v \neq k}}^{}a_{u,v}^{\nu_{u,v}})\sum_{N+M=m}f_Ng_M\Big)|.\] 
 Notice that,  since $N,M \leq m$ and  $f_N,g_M \in \mathbb{Z}_p$ from \ref{eq:fndefinition} and \ref{eq:gmdefinition}, the quantity
  $(a_{k,k}^{\nu_{k,k}}\prod_{\underset{u>v}{u,v \neq k}}^{}a_{u,v}^{\nu_{u,v}}\sum_{N+M=m}f_Ng_M)$ has finite sum and product and hence lies in $\mathbb{Z}_p$. Hence, \[|\Big(\sum_{\nu}c_{\nu}(a_{k,k}^{\nu_{k,k}}\prod_{\underset{u>v}{u,v \neq k}}^{}a_{u,v}^{\nu_{u,v}})\sum_{N+M=m}f_Ng_M\Big)| \rightarrow 0\] 
 as  $|c_{\nu}| \rightarrow 0$ with $\nu \rightarrow \infty$. This gives the analyticity of the action $f \rightarrow f(\mathcal{C})$
We treat the analyticity of the character $\chi$ in general.
 Write  $\chi=(\chi_1,...,\chi_n)$, $\chi_i(1+pu_i)=e^{c_i\log(1+pu_i)}$ for $c_i \in K$, $u_i$ close to $0$, $i \in [1,n]$. The exponential is analytic (in $K$) in the domain $v_p(z) > \frac{e}{p-1}$ where $e=e(K)$ and $v_p$ is the normalized valuation, $v_p(p)=1$. Now, \[v_p(c_i\log(1+pu_i))=v_p(c_i)+1+v_p(u_i).\] So we must have $v_p(c_i)+1 >\frac{e}{p-1}$, i.e. 
\begin{align}\label{eq:unramifiedK}
v_p(c_i)&>\frac{e}{p-1}-1,\\
&=\frac{-p}{p-1}  \text{ (if } K \text{ is unramified)}.
\end{align}
We say that $\chi$ is "analytic" if and only if $c_i$'s verify these conditions and in the rest of this text we assume that our character $\chi$ is analytic.
It is easy to see that if $\chi$ is analytic, then $\chi(1,...,1+p\xi_k,...,1)$ is an  analytic function on $\xi_k$. The  character

\begin{align*}
\chi(1,...,1+p\xi_k,...,1)&=\chi_k(1+p\xi_k)=\sum_{n=0}^{\infty}c_n(1+p\xi_k)^n \qquad (\text{since } \chi_k \text{ is analytic})\\
&=\sum_{n=0}^{\infty}c_n\sum_{u=0}^n{{n}\choose{u}}p^u\xi_k^u\\
&=\sum_{u=0}^{\infty}\xi_k^u\Big(p^u\sum_{n\geq u}^{\infty}c_n{{n}\choose{u}}\Big).
\end{align*} 
The norm of the coefficient of $\xi_k^u$ is  $|p^u\sum_{n\geq u}^{\infty}c_n{{n}\choose{u}}|$ which goes to $0$ as $|c_n| \rightarrow 0$ with $n \rightarrow \infty$.
Thus, we have shown
\begin{lemma}\label{lem:lemmadiagonal1}
Under the  hypothesis \ref{eq:unramifiedK}, for each $k \in [1,n]$, the action of the one-parameter diagonal subgroup $(t_kE_{k,k}+\sum_{i=1,i \neq k}^nE_{i,i})$ of $G$ on $\mathcal{A}(\mathbb{Z}_p^{\frac{n(n-1)}{2}},K)$ given by \ref{eq:analyticexpressionone1} is an analytic action.
\end{lemma}
For $y \in \mathbb{Z}_p$ and $i>j$, $i,j$ fixed between $1,...,n$, the action of the $1$-parameter (rigid-analytic) subgroup $(1+yE_{i,j})$ on $f(A)$, given by \ref{eq:Gaction} is 
\begin{align}\label{eq:lowertriangularanalytic}
(1+yE_{i,j})f(A)&=f\Big((1+yE_{i,j})^{-1}A \Big)=f\Big((1-yE_{i,j})A\Big),\\
&=f\Big((1-yE_{i,j})(\sum\limits_{\substack{k \geq l\\k,l \in [1,n]}}a_{k,l}E_{k,l})\Big),\\
&=f\Big( \sum\limits_{\substack{k \geq l\\k,l \in [1,n]}}a_{k,l}E_{k,l}-\sum_{l=1,...,j}ya_{j,l}E_{i,l}\Big):=f(\mathcal{B}),\label{arparina}
\end{align}
where $\mathcal{B}$ is the matrix $\sum\limits_{\substack{k \geq l\\k,l \in [1,n]}}a_{k,l}E_{k,l}-\sum_{l=1}^jya_{j,l}E_{i,l}$.
One can easily see that the matrix $\mathcal{B}=(b_{u,v})$ is lower unipotent and differs from matrix $A$ only in the first $j$ entries of its $i^{th}$ row. In particular, $b_{i,v}=a_{i,v}-ya_{j,v}$ for all $v \in [1,j], (a_{j,j}=1)$ and all other $b_{u,v}$ are the same as $a_{u,v}$ (recall that $A$ is lower unipotent). 

Now, let $f$ be a globally analytic function on $U$ as in \ref{eq:fanalytic}. That is $f(a)=\sum_{\nu \in \mathbb{N}^d}c_{\nu}a^{\nu}$ with $a^{\nu}=a_{2,1}^{\nu_{2,1}}\cdots a_{n,n-1}^{\nu_{n,n-1}}$ and $|c_{\nu}| \rightarrow 0$. Then, we have to show that $(1+yE_{i,j})f=f(\mathcal{B})$ gives an analytic map 
\begin{align*}
\mathbb{Z}_p &\rightarrow \mathcal{A}(U,K)\\
y &\rightarrow (1+yE_{i,j})f=f(\mathcal{B}).
\end{align*}
The power series \[f(\mathcal{B})=\sum_{\nu}c_{\nu}\Big((\prod_{\underset{u=i \implies v>j}{u>v}}^{}a_{u,v}^{\nu_{u,v}})(\prod_{k=1}^j(a_{i,k}-ya_{j,k})^{\nu_{i,k}})\Big).\]
For each $k\in [1,j]$, inserting the expansion \[(a_{i,k}-ya_{j,k})^{\nu_{i,k}}=\sum_{m_{i,k}=0}^{\nu_{i,k}}{\nu_{i,k}\choose{m_{i,k}}}y^{m_{i,k}}(-a_{j,k})^{m_{i,k}}a_{i,k}^{\nu_{i,k}-m_{i,k}}\] into $f(\mathcal{B})$ we obtain, for $M=\prod_{k=1}^j\nu_{i,k}, |m|=m_{i,1}+\cdots +m_{i,j}$,
\begin{align*}
f(\mathcal{B})&=\sum_{\nu}c_{\nu}\Big((\prod_{\underset{u=i \implies v>j}{u>v}}^{}a_{u,v}^{\nu_{u,v}})(\prod_{k=1}^j(\sum_{m_{i,k}=0}^{\nu_{i,k}}{\nu_{i,k}\choose{m_{i,k}}}y^{m_{i,k}}(-a_{j,k})^{m_{i,k}}a_{i,k}^{\nu_{i,k}-m_{i,k}}))\Big)\\
&=\sum_{\nu}c_{\nu}\Big((\prod_{\underset{u=i \implies v>j}{u>v}}^{}a_{u,v}^{\nu_{u,v}})(\sum_{N=0}^My^N(\sum_{\underset{m_{i,\star}\in [0,\nu_{i,\star}]}{|m|=N}}\prod_{k=1}^j{\nu_{i,k}\choose{m_{i,k}}}(-a_{j,k})^{m_{i,k}}a_{i,k}^{\nu_{i,k}-m_{i,k}}))\Big)\\
&=\sum_{\nu}c_{\nu}\Big(\sum_{N=0}^My^N((\prod_{\underset{u=i \implies v>j}{u>v}}^{}a_{u,v}^{\nu_{u,v}})\sum_{\underset{m_{i,\star}\in [0,\nu_{i,\star}]}{|m|=N}}\prod_{k=1}^j{\nu_{i,k}\choose{m_{i,k}}}(-a_{j,k})^{m_{i,k}}a_{i,k}^{\nu_{i,k}-m_{i,k}})\Big)\\
&=\sum_{N=0}^{M}y^N\Big(\sum_{\nu}c_{\nu}((\prod_{\underset{u=i \implies v>j}{u>v}}^{}a_{u,v}^{\nu_{u,v}})\sum_{\underset{m_{i,\star}\in [0,\nu_{i,\star}]}{|m|=N}}\prod_{k=1}^j{\nu_{i,k}\choose{m_{i,k}}}(-a_{j,k})^{m_{i,k}}a_{i,k}^{\nu_{i,k}-m_{i,k}})\Big)\\
&=\sum_{N=0}^{M}y^Nf_N
\end{align*}
where $$f_N:=\sum_{\nu}c_{\nu}\Big((\prod_{\underset{u=i \implies v>j}{u>v}}^{}a_{u,v}^{\nu_{u,v}})\sum_{\underset{m_{i,\star}\in [0,\nu_{i,\star}]}{|m|=N}}\prod_{k=1}^j{\nu_{i,k}\choose{m_{i,k}}}(-a_{j,k})^{m_{i,k}}a_{i,k}^{\nu_{i,k}-m_{i,k}}\Big).$$
Define \[s(N, \nu):=\Big((\prod_{\underset{u=i \implies v>j}{u>v}}^{}a_{u,v}^{\nu_{u,v}})\sum_{\underset{m_{i,\star}\in [0,\nu_{i,\star}]}{|m|=N}}\prod_{k=1}^j{\nu_{i,k}\choose{m_{i,k}}}(-a_{j,k})^{m_{i,k}}a_{i,k}^{\nu_{i,k}-m_{i,k}}\Big)\] such that $f_N=\sum_{\nu}c_{\nu}s(N,\nu)$. Notice that since 
$m_{i,k}\leq \nu_{i,k}$ for all $k \in [1,j]$, the sum in $s(N, \nu)$ is a finite sum and thus $s(N, \nu)$ lies in $\mathbb{Z}_p$. Therefore,  the norm of the coefficient of $y^N$ is $|f_N|=|\sum_{\nu}c_{\nu}s(N,\nu)|$ which goes to $0$ as $|c_{\nu}| \rightarrow 0$ with $|\nu| \rightarrow \infty$. This gives the analyticity of the map $y \rightarrow (1+E_{i,j})f=f(\mathcal{B})$.

Therefore, we have shown,
\begin{lemma}\label{lem:lemmalower1}
For $y\in \mathbb{Z}_p$ and $i>j$, the action of the lower unipotent (rigid-analytic) $1$-parameter subgroup $(1+yE_{i,j})$  of $G$ on $f \in \mathcal{A}(\mathbb{Z}_p^{\frac{n(n-1)}{2}},K)$, given by \ref{eq:lowertriangularanalytic} is an analytic action.
\end{lemma}
It remains to check the analyticity of the action \ref{eq:Gaction} by triangular superior matrices of the form $(1+yE_{i,j})$ for $i<j,i,j \in [1,n], y \in p\mathbb{Z}_p$. Recall that the action of $(1+yE_{i,j})$ on $f \in I_{\loc}$ given by \ref{eq:Gaction}, is 
\begin{equation*}
(1+yE_{i,j})f(A)=f\Big((1+yE_{i,j})^{-1}A \Big)=f\Big((1-yE_{i,j})A\Big).
\end{equation*}
Recall the  action of $Q_0$ given by \ref{eq:chiaction}, that is, $f(gb)\equiv \chi(b^{-1})f(g)$ with $b \in Q_0$. Hence, our objective is to write the matrix $(1-yE_{i,j})A$ as the product of two matrices $X$ and $Z$ with $X \in U$ and $Z \in Q_0$, that is:
\begin{equation*}
(1-yE_{i,j})A=XZ,
\end{equation*}
where $X$ is a lower unipotent matrix with entries in $\mathbb{Z}_p$ and $Z$ is a upper triangular matrix with diagonal elements in $1+p\mathbb{Z}_p$ and such that the elements above the diagonal have entries in $p\mathbb{Z}_p$. 
\begin{lemma}\label{lem:biglemma}
For $i <j$ and $ y \in p\mathbb{Z}_p$, there exists a unique matrix decomposition $(1-yE_{i,j})A=XZ$ with $X=(x_{k,l})_{k,l}\in U$ and $Z=(z_{r,s})_{r,s} \in Q_0$. Also, 
\begin{enumerate}
 \item all the diagonal elements $z_{r,r}$  of $Z$  are of the form  $\frac{1-yh_{{r,r}}(y,a)}{1-yg_{{r,r}}(y,a)}$,
\item  all the  elements $z_{r,s}$,  for $r<s$, of $Z$   are of the form $\frac{yh_{{r,s}}(y,a)}{1-yg_{{r,s}}(y,a)}$,
\item all the elements $x_{k,l}$  with $k>l$  of the lower triangular unipotent matrix $X$  are of the form 
$\frac{h_{{k,l}}(y,a)}{1-yg_{{k,l}}(y,a)},$
\end{enumerate}
where $h_{\star,\star}(y,a)$ and $g_{\star, \star}(y,a)$ are polynomial functions with integral coefficients in $y$ and $a_{2,1},a_{3,1},a_{3,2},...,a_{n,n-1}$ (entries of the lower unipotent matrix $A$).
\end{lemma}
\begin{proof}
We prove the lemma by an easy inductive argument. The base case $n=2$ is clear from the matrix equation
\begin{align*}
 \left( \begin{array}{cc}
1 & -y \\
0 & 1
\end{array} \right)
\left( \begin{array}{cc}
1 & 0 \\
a_{2,1} & 1
\end{array} \right)
&=\left( \begin{array}{cc}
1-ya_{2,1} & -y \\
a_{2,1} & 1
\end{array} \right)
=\left( \begin{array}{cc}
1 & 0 \\
x_{2,1} & 1
\end{array} \right)
\left( \begin{array}{cc}
z_{1,1} & z_{1,2} \\
0 & z_{2,2}
\end{array} \right) \\
&=\left( \begin{array}{cc}
z_{1,1} & z_{1,2} \\
z_{1,1}x_{2,1} & z_{2,2}+z_{1,2}x_{2,1}
\end{array} \right)
\end{align*}
with $x_{2,1}=\frac{a_{2,1}}{1-ya_{2,1}}, z_{1,1}=1-ya_{2,1},z_{1,2}=-y,z_{2,2}=\frac{1}{1-y{a_{2,1}}}$. Assume, by induction hypothesis that our lemma is true for $GL(n-1)$. We show it for $GL(n)$.
Let us first suppose that $i>1$, that is 
\[
(1-yE_{i,j})=
\left(
\begin{array}{c|c}
  1 & 0 \cdots 0 \\ \hline
  0 & \raisebox{-15pt}{{\mbox{{$1-yE^{\prime}$}}}} \\[-4ex]
  \vdots & \\[-0.5ex]
  0 &
\end{array}
\right)
 \text{ } (\text{with some elementary matrix } E^{\prime}, 1-yE^{\prime}\in GL(n-1))
\]
The matrix $A$, being lower unipotent, can be written in the following block form:
\[
A=
\left(
\begin{array}{c|c}
  1 & 0 \cdots 0 \\ \hline
  a_{2,1} & \raisebox{-15pt}{{\mbox{{$A^{\prime}$}}}} \\[-4ex]
  \vdots & \\[-0.5ex]
  a_{n,1} &
\end{array}
\right)
 \text{ } (\text{with } A^{\prime} \in GL(n-1))
\]
Setting $\underbar{a}$ to be the column vector $
\left(
\begin{array}{c}
a_{2,1}\\
a_{3,1}\\
\vdots\\
a_{n,1}
\end{array}
\right),$ 

\[
(1-yE_{i,j})A=
\left(
\begin{array}{c|c}
  1 & 0 \cdots 0 \\ \hline
  0 & \raisebox{-15pt}{{\mbox{{$1-yE^{\prime}$}}}} \\[-4ex]
  \vdots & \\[-0.5ex]
  0 &
\end{array}
\right)
\left(
\begin{array}{c|c}
  1 & 0 \cdots 0 \\ \hline
  a_{2,1} & \raisebox{-15pt}{{\mbox{{$A^{\prime}$}}}} \\[-4ex]
  \vdots & \\[-0.5ex]
  a_{n,1} &
\end{array}
\right)
=\left(
\begin{array}{c|c}
1 & 0 \\
\hline
(1-yE^{\prime})\underbar{a} & (1-yE^{\prime})A^{\prime}
\end{array}
\right)
\]
We want to decompose the above matrix in the form 

\[
\left(
\begin{array}{c|c}
1 & 0 \\
\hline
(1-yE^{\prime})\underbar{a} & (1-yE^{\prime})A^{\prime}
\end{array}
\right)
=\left(
\begin{array}{c|c}
  1 & 0 \cdots 0 \\ \hline
  x_{2,1} & \raisebox{-15pt}{{\mbox{{$X^{\prime}$}}}} \\[-4ex]
  \vdots & \\[-0.5ex]
  x_{n,1} &
\end{array}
\right)
\left(
\begin{array}{c|c}
  z_{1,1} & z_{1,2} \cdots z_{1,n} \\ \hline
  0 & \raisebox{-15pt}{{\mbox{{$Z^{\prime}$}}}} \\[-4ex]
  \vdots & \\[-0.5ex]
  0 &
\end{array}
\right)
\]
with $x_{2,1},...,x_{n,1} \in \mathbb{Z}_p$, $z_{1,1} \in 1+p\mathbb{Z}_p$ and $z_{1,2},...,z_{1,n} \in p\mathbb{Z}_p$. 
Denote $\underbar{z}$  to be the row vector $[z_{1,2},...,z_{1,n}]$,  $\underbar{x}$ to be the column vector $
\left(
\begin{array}{c}
x_{2,1}\\
x_{3,1}\\
\vdots\\
x_{n,1}
\end{array}
\right)$. Hence, we want to solve 
\[
\left(
\begin{array}{c|c}
1 & 0 \\
\hline
(1-yE^{\prime})\underbar{a} & (1-yE^{\prime})A^{\prime}
\end{array}
\right)
=\left(
\begin{array}{c|c}
z_{1,1} & \underbar{z} \\
\hline
z_{1,1}\underbar{x} & \underbar{x}\cdot \underbar{z}+X^{\prime}Z^{\prime}
\end{array}
\right)
\]
So we must have 
\begin{enumerate}
\item $z_{1,1}=1$,
\item $\underbar{z}=0$,
\item $z_{1,1}\underbar{x}=\underbar{x}=(1-yE^{\prime})\underbar{a}$ (using $z_{1,1}=1$ from (1)),
\item $\underbar{x} \cdot \underbar{z}+X^{\prime}Z^{\prime}=X^{\prime}Z^{\prime}=(1-yE^{\prime})A^{\prime}$   (as $\underbar{z}=0$ from (2)).
\end{enumerate}
By the induction hypothesis, we can find $X^{\prime}$ and $Z^{\prime}$ satisfying $(4)$ with entries in as lemma \ref{lem:biglemma}. Also, $(3)$ is of the form 
\[
\begin{pmatrix}
1\\
&&&&-y\\
&&&&\ddots\\
&&&&&1\\
\end{pmatrix}
\left(
\begin{array}{c}
a_{2,1}\\
a_{3,1}\\
\vdots\\
a_{n,1}
\end{array}
\right)
=\left(
\begin{array}{c}
x_{2,1}\\
x_{3,1}\\
\vdots\\
x_{n,1}
\end{array}
\right)
\] 

Clearly, we can solve $x_{2,1},..,x_{n,1}$ from the above matrix equation satisfying lemma \ref{lem:biglemma} and in fact the solutions do not have any denominators.

So by induction we are reduced to the case $i=1$, that is, when 
\[
(1-yE_{i,j})=
\left(
\begin{array}{c|c}
  1 &  0 \cdots -y \cdots  0 \\ \hline
  0 & \raisebox{-15pt}{{\mbox{{$1$}}}} \\[-4ex]
  \vdots & \\[-0.5ex]
  0 &
\end{array}
\right)
\]
Our goal is to solve, for $X$ and $Z$, the following matrix equation:
\begin{equation}\label{eq:XZRHS}
(1-yE_{1,j})A=XZ.
\end{equation}
Expanding right hand side of \ref{eq:XZRHS}, we obtain
\begin{align*}
 \mathcal{B}=(b_{u,v})_{u,v}&=XZ=\Big(1+\sum\limits_{\substack{k\in[1,n]\\l \in [1,k-1]}}x_{k,l}E_{k,l}\Big)\Big(\sum\limits_{\substack{r \in [1,n]\\s \in [r,n]}}z_{r,s}E_{r,s}\Big)\\
 &=\sum\limits_{\substack{r \in [1,n]\\s \in [r,n]}}z_{r,s}E_{r,s}+\sum\limits_{\substack{k \in [1,n]\\r \in [1,k-1]\\s \in [r,n]}}x_{k,r}z_{r,s}E_{k,s}.
 \end{align*}
 Therefore,
  \begin{equation}\label{eq:duv}
      b_{u,v}=
      \begin{cases}
        \sum_{r=1}^vx_{u,r}z_{r,v}, & \text{if}\ u>v \\
        z_{u,v}+\sum_{r=1}^{u-1}x_{u,r}z_{r,v}, & \text{if}\ u \leq v 
      \end{cases}
    \end{equation}
  
Recall that our matrix $A=\sum_{k \geq l}a_{k,l}E_{k,l}$ is lower unipotent, that is $a_{k,k}=1 $ for all $k \in [1,n]$ and $a_{k,l}=0$ for $k<l$. Expanding the left hand side of \ref{eq:XZRHS}, we obtain
\begin{align*}
(1-yE_{1,j})A&=(1-yE_{1,j})(\sum_{k \geq l}a_{k,l}E_{k,l})= \sum_{k \geq l}a_{k,l}E_{k,l}-\sum_{l=1}^jya_{j,l}E_{1,l},\\
&=\sum\limits_{\substack{k \in [1,n]\\l\in [1,k]\\k \neq 1}}a_{k,l}E_{k,l}+\sum_{l=2}^j(-ya_{j,l})E_{1,l}+(1-ya_{j,1})E_{1,1} .
\end{align*}
Note that the $1^{st}$ row of the matrix $(1-yE_{1,j})A$ is $$\sum_{l=2}^j(-ya_{j,l})E_{1,l}+(1-ya_{j,1})E_{1,1}.$$
From \ref{eq:XZRHS}, the matrices $(1-yE_{1,j})A$ and $\mathcal{B}=(b_{u,v})_{u,v}$ are equal. Thus, equating $b_{u,v}$ from \ref{eq:duv} with the above expression of the matrix $(1-yE_{1,j})A$, we obtain the following equations (with the convention that $x_{k,l}=0$ for $k \leq l$ and $z_{r,s}=0$ for $r>s$):

\begin{enumerate}
\item for $u \neq 1$ and $u>v$, $b_{u,v}=\sum_{r=1}^vx_{u,r}z_{r,v}=a_{u,v},$
\item for $u \neq 1$ and $u=v$, $b_{u,v}=z_{u,u}+\sum_{r=1}^{u-1}x_{u,r}z_{r,v}=a_{u,u}=1$,
\item for $u \neq 1$ and $u<v$, $b_{u,v}=z_{u,v}+\sum_{r=1}^{u-1}x_{u,r}z_{r,v}=a_{u,v}=0$,
\item for $u=v=1$, $b_{1,1}=z_{1,1}=1-ya_{j,1}$,
\item for $u=1$ and $u <v$, $b_{1,v}=z_{1,v}=-ya_{j,v}$.
\end{enumerate}
Note that in (5), for $v>j$, $b_{1,v}=-ya_{j,v}=0$ (as $A$ is lower unipotent).
Setting $v=1$ in formula $(1)$, for $u\in [2,n]$,  we obtain 
\begin{equation}\label{eq:solvex}
x_{u,1}=\frac{a_{u,1}}{z_{1,1}}=\frac{a_{u,1}}{1-ya_{j,1}}  \text{ } (\text{as } z_{1,1}=1-ya_{j,1} \text{ from } (4)).
\end{equation}
Now, let $C=(c_{k,l})_{k,l}=(1-yE_{1,j})A$ and $\mathcal{B}=(b_{u,v})_{u,v}$ as above. We proceed by equating, in the $1^{st}$ stage, the first row of the matrix $\mathcal{B}$ with the first row of the matrix $C$, starting from the leftmost entry (i.e. given by equations $(4)$ and $(5)$ above) and solve for $z_{\star,\star}$. Then in the next stage (say stage $1+\frac{1}{2}$) we equate the first column of the  matrix $\mathcal{B}$ with the first column of the matrix $C$ starting from the uppermost entry ($b_{2,1}=c_{2,1}$)  and solve for $x_{\star,\star}$ (i.e. those given by \ref{eq:solvex}). In the second stage we do the same with the second row and in the  stage $2+\frac{1}{2}$ we equate the second column of the matrix $\mathcal{B}$ with $C$ (given by (1), (2) and (3)) and proceed like this until the last $n^{th}$ stage. Our objective is to solve $x_{\star,\star}$ and $z_{\star,\star}$ while equating the matrix $\mathcal{B}$ with $C$ and show $(1),(2)$ and $(3)$ of lemma \ref{lem:biglemma}. We prove this by induction. 

Assume, by induction hypothesis, at the $m^{th}$ and the  stage $m+\frac{1}{2}$ ($1 \leq m < n$), that we have found $x_{k,l}$ for $k\in [2,n],l \in [1,m],k>l$ and $z_{r,s}$ for $r \in [1,m],s \in [1,n],r \leq s$ having the form  $(1), (2)$ and $(3)$ of lemma \ref{lem:biglemma}. Then, at the $(m+1)^{th}$ stage  we have to equate $b_{m+1,v}=c_{m+1,v}$ for $v \in [m+1,n]$. Equating $b_{m+1,m+1}=c_{m+1,m+1}$, we deduce, by formula (2) after \ref{eq:duv}, that
\begin{align*}
z_{m+1,m+1}&=1-\sum_{r=1}^mx_{m+1,r}z_{r,m+1},\\
&=1-\frac{yh_1(y,a)}{1-yg_1(y,a)} \text{ } (\text{by induction hypothesis}),\\
&=\frac{1-y(h_1+g_1)}{1-yg_1},
\end{align*}
 for some polynomial functions $h_1(y,a)$ and $g_1(y,a)$ with integral coefficients in $y$ and $a_{2,1},a_{3,1},a_{3,2},...,a_{n,n-1}$.
 
 Similarly, equating $b_{m+1,v}=c_{m+1,v}$ for $v \in [m+2,n]$, we obtain, by formula (3) after \ref{eq:duv}, that
 \begin{align*}
 z_{m+1,v}&=-\sum_{r=1}^mx_{m+1,r}z_{r,v},\\
 &=\frac{-yh_2(y,a)}{1-yg_2(y,a)} \text{ } (\text{again by induction hypothesis}),
 \end{align*}
At the  stage $(m+1)+\frac{1}{2}$ we have to equate $b_{u,m+1}=c_{u,m+1}$ for all $u \in [m+2,n]$. So, by formula (1) after \ref{eq:duv}, we get
\begin{align*}
x_{u,m+1}z_{m+1,m+1}&=a_{u,m+1}-\sum_{r=1}^mx_{u,r}z_{r,m+1},\\
&=a_{u,m+1}-\frac{yh_3(y,a)}{1-yg_3(y,a)} \text{ } (\text{again by induction hypothesis}),\\
&=\frac{h_4(y,a)}{1-yg_3(y,a)},
\end{align*}
for some $h_4$ and $g_3$ with integral coefficients. Therefore, 
\begin{equation*}
x_{u,m+1}=\frac{h_4(y,a)}{1-yg_3(y,a)}\cdot \frac{1}{z_{m+1,m+1}}=\frac{h_4(y,a)}{1-yg_3(y,a)}\frac{1-yg_1}{1-y(h_1+g_1)}=\frac{h_5(y,a)}{1-yg_5(y,a)},
\end{equation*}
with polynomials $h_5$ and $g_5$ having integral coefficients.
This completes our induction argument and finishes the proof of lemma \ref{lem:biglemma}.

\end{proof}

Now, let $f \in I_{\loc}$. Then, by lemma \ref{lem:biglemma}, the action of 
$(1+yE_{i,j})$ on $f$ is given by 
\begin{equation}\label{eq:analyticactionthird}
(1+yE_{i,j})f(A)=f(X)\chi(z_{1,1}^{-1},...,z_{n,n}^{-1}) \text { } (\text{ with } X \text{ and } z_{\star,\star} \text{ as in  lemma } \ref{lem:biglemma}).
\end{equation}
Recall that for $|v|<1$, we have 
\begin{equation*}
(1-v)^{-m}=\sum_{q=0}^{\infty}\binom{m+q-1}{q}v^q.
\end{equation*}
Assume now that $f \in \mathcal{A}(\mathbb{Z}^{\frac{n(n-1)}{2}},K)$ is a globally analytic function. Thus, $f$ is an element in the Tate algebra of $U$ with $\frac{n(n-1)}{2}$ variables. 
In order to show that the action of $(1+yE_{i,j})$ on $f \in \mathcal{A}(U,K)$, given by equation \ref{eq:analyticactionthird}, is globally analytic we have to show that \[\prod_{r=1}^n\chi_r(z_{r,r}^{-1})f(X)\] is a globally analytic function in $y$.
\begin{lemma}\label{lem:ifftofx}
If the action $f \rightarrow g$, $g(A)=f(X)$ where $A=XZ$ is globally analytic, then $f \rightarrow \prod_{r=1}^n\chi_r(z_{r,r}^{-1})g$ is globally analytic.  
\end{lemma}
\begin{proof}
With \ref{eq:unramifiedK}, our character $\chi$ is analytic. Hence, 
\begin{align*}
\chi_r(z_{r,r}^{-1})&=\chi_r\Big(\frac{1-yg_{{r,r}}(y,a)}{1-yh_{{r,r}}(y,a)}\Big) \qquad \qquad (\text{ from lemma } \ref{lem:biglemma})\\
&=\sum_{n=0}^{\infty}c_n\Big(\frac{1-yg_{{r,r}}(y,a)}{1-yh_{{r,r}}(y,a)}\Big)^n  \qquad (\text{ for } |c_n| \rightarrow 0).
\end{align*}
We are reduced to show that $(y,a) \rightarrow \frac{1-yg_{{r,r}}(y,a)}{1-yh_{{r,r}}(y,a)}$ is analytic in $y$ and this is true because \[\frac{1-yg_{{r,r}}(y,a)}{1-yh_{{r,r}}(y,a)}=(1-yg_{{r,r}}(y,a))\Big(\sum_{n=0}^{\infty}(yh_{{r,r}}(y,a))^n\Big).\]
Hence we have shown the lemma.
\end{proof}
Therefore, with lemma \ref{lem:ifftofx}, to prove that the action of $(1+yE_{i,j})$ on $f \in \mathcal{A}(U,K)$, given by equation \ref{eq:analyticactionthird}, is globally analytic, we only need to show that the action $f \rightarrow g$, $g(A)=f(X)$ where $A=XZ$ is globally analytic.
\begin{lemma}
The action $f \rightarrow g$, $g(A)=f(X)$ where $A=XZ$ is globally analytic.
\end{lemma}
\begin{proof}
Recall that the lower unipotent matrix $X$ is $((x_{k,l})_{k,l})$ with $x_{k,l}=\frac{h_{k,l}(y,a)}{1-yg_{k,l}(y,a)}$ given by lemma \ref{lem:biglemma}. Write 
\begin{align*}
x_{k,l}&=h_{k,l}(y,a) \sum_{n=0}^{\infty}y^ng_{k,l}(y,a)^n\\
&=\sum_{n=0}^{\infty}y^ng_{n,k,l}(y,a).
\end{align*} 
Since $f$ is analytic, $f=\sum_Nf_Nx^N$ with $N=(N_{k,l})\in \mathbb{Z}_p^{\frac{n(n-1)}{2}},x^N=\prod_{k>l}x_{k,l}^{N_{k,l}}$. The norm $|f_N| \rightarrow 0$ as $N \rightarrow \infty$. Then,
\[f(X)=f((x_{k,l})_{k,l})=\sum_Nf_N\prod_{\underset{k>l}{k,l}} (\sum_{n=0}^{\infty}y^ng_{n,k,l}(y,a))^{N_{k,l}}.\]
As \[(\sum_{n=0}^{\infty}y^ng_n)^M=\sum_{v \geq 0}{y^{v}}\sum_{v_1+\cdots+v_M=v}g_{v_1}\cdots g_{v_M},\] we obtain that
\[f(X)=\sum_{N=(N_{k,l})}f_N\prod_{k,l}\sum_{v \geq 0}y^v(\sum_{v_1+\cdots + v_{N_{k,l}}=v}g_{v_1,k,l}\cdots g_{v_{N_{k,l}},k,l}).\]
Define $a_{k,l}(v)=(\sum\limits_{v_1+\cdots + v_{N_{k,l}}=v}g_{v_1,k,l}\cdots g_{v_{N_{k,l}},k,l})$ then,
\begin{align*}
f(X)&=\sum_{N=(N_{k,l})}f_N\prod_{k,l}\sum_{v \geq 0}y^va_{k,l}(v)\\
&=\sum_{N=(N_{k,l})}f_N\sum_{v \geq 0}^{\infty}y^v\sum_{\sum {v_{k,l=v}}}\prod_{k,l}a_{k,l}(v_{k,l})\\
&=\sum_{N=(N_{k,l})}f_N\sum_{v \geq 0}^{\infty}y^v\sum_{\sum {v_{k,l=v}}}\prod_{k,l}\sum\limits_{v_1+\cdots + v_{N_{k,l}}=v_{k,l}}g_{v_1,k,l}\cdots g_{v_{N_{k,l}},k,l}.
\end{align*}
The coefficient of $y^v$ is \[\sum_{N=(N_{k,l})}f_N\sum_{\sum {v_{k,l=v}}}\prod_{k,l}\sum\limits_{v_1+\cdots + v_{N_{k,l}}=v_{k,l}}g_{v_1,k,l}\cdots g_{v_{N_{k,l}},k,l}:=\sum_{N=(N_{k,l})}f_Ns_N\]
where $$s_N:=\sum_{\sum {v_{k,l=v}}}\prod_{k,l}\sum\limits_{v_1+\cdots + v_{N_{k,l}}=v_{k,l}}g_{v_1,k,l}\cdots g_{v_{N_{k,l}},k,l}.$$
Here $N_{k,l}$ is finite and $v_{k,l} \leq v$ and hence the sum $s_N$ is a finite sum in $\mathbb{Z}_p$. As, with $N \rightarrow \infty$, $|f_N|\rightarrow 0$, we obtain that $|\sum_Nf_Ns_N| \rightarrow 0$ and this completes the proof. 
\end{proof}
This shows the analyticity of the action given by \ref{eq:analyticactionthird}. So we have shown
\begin{lemma}\label{lem:lemmaupper2}
For $y\in p\mathbb{Z}_p$ and $i<j$, the action of the upper unipotent (rigid-analytic) $1$-parameter subgroup $(1+yE_{i,j})$  of $G$ on $f \in \mathcal{A}(\mathbb{Z}_p^{\frac{n(n-1)}{2}},K)$, given by \ref{eq:analyticactionthird} is an analytic action.
\end{lemma}
Note that, by section \ref{sub:sectionbigger}, the vector space of locally analytic functions of principal series $$\ind_{P_0}^B(\chi)_{\loc}=\{f \in \mathcal{A}_{\loc}(B,K):f(gb)=\chi(b^{-1})f(g),b \in P_0,g \in B\}$$ is isomorphic to the vector space of the locally analytic functions \\
$I_{\loc}\cong \mathcal{A}_{\loc}(\mathbb{Z}_p^{\frac{n(n-1)}{2}},K)$. Denote by $\ind_{P_0}^B(\chi)$ the space of globally analytic vectors of $\ind_{P_0}^B(\chi)_{\loc}$ which is $\mathcal{A}:=\mathcal{A}(\mathbb{Z}_p^{\frac{n(n-1)}{2}},K)$. 

Also,  the representation on $\mathcal{A}$ is admissible: indeed, $\mathcal{A}$ is a subspace of $\mathcal{A}(G)$ defined by the conditions $f(gb) = \chi(b^{-1})f(g)$ ($f$ is then analytic on $G$ since $\chi$ is so) and this is a closed subspace. 
Thus by lemmas (\ref{lem:sufficeaction}- \ref{lem:lemmalower1}) and lemma \ref{lem:lemmaupper2} we have shown the following theorem.
\begin{theorem}\label{thmmainglo}
Assume $p>n+1$. Let $\chi$ be an analytic character of $T_0$ (cf. \ref{eq:unramifiedK}). The action of $G$ on the induced principal series $\ind_{P_0}^B(\chi)$ is a globally analytic action. Moreover, the globally analytic representation of $G$ on $\ind_{P_0}^B(\chi)$ is admissible in the sense of Emerton (\cite{Emertonbook}, \cite[sec. $2.3$]{Clozel1}).
\end{theorem}
Recall that $\chi=(\chi_1,..,\chi_n)$ where $\chi_i(1+pu_i)=e^{c_i\log(1+pu_i)}$ for $c_i \in K$, $u_i$ close to $0$, $i \in [1,n]$. Also, recall from \ref{eq:fanalytic} that $f \in \mathcal{A}$ implies that $f(A)=\sum_{\nu \in \mathbb{N}^d}c_{\nu}a^{\nu}$ with $|c_{\nu}|\rightarrow 0$ as $|\nu|=\nu_{2,1}+\nu_{3,1}+\cdots +\nu_{n,n-1} \rightarrow \infty$.

In the following, we will have conditions on the character $\chi$ such that the globally analytic representation of $G$ on $\mathcal{A}$ is irreducible.  

Let $\mu$ be the linear form from the Lie algebra of the torus $T_0$ to $K$ given by 
  $$\mu=(-c_1,...,-c_n):Diag(t_1,...,t_n)\mapsto \sum_{i=1}^n-c_it_i$$
 where $t=(t_i) \in \Lie(T_0)$. For negative root $\alpha =(i,j), i>j$, let $H_{(i,j)}$ be the matrix $E_{i,i}-E_{j,j}$ where $E_{i,i}$ is the standard elementary matrix.

\begin{theorem}\label{thm:holomorphic}
Let $c_i$'s satisfy  \ref{eq:unramifiedK} and $p>n+1$, then the globally analytic representation $\mathcal{A}\cong \ind_{P_0}^B(\chi)$ of $G$ is topologically irreducible if and only if for all $\alpha=(i,j) \in \Phi^-$,  $-\mu(H_{\alpha=(i,j)}) + i-j \notin \{1,2,3,...\}$.
\end{theorem}

Assume $\mathcal{X} \subset \mathcal{A}$ is a closed $G$-invariant subspace. Let $\Phi,\Phi^-,\Phi^+,\Pi$ be the roots, negative roots, positive roots and simple roots respectively associated to $G$.  Consider $f \in \mathcal{A}$. Then, from \ref{eq:fanalytic}, 
$$f=\sum_{\nu \in \mathbb{N}^d}c_{\nu}a^{\nu}$$ where $d=\frac{n(n-1)}{2}, c_{\nu} \in K$, $|c_{\nu}| \rightarrow 0$ as $|\nu|:=\sum_{\alpha \in \Phi^-} \nu_{\alpha} \rightarrow \infty$. \\
Here, $\nu=(\nu_{\alpha}, \alpha \in \Phi^-) \in \N^d$ and $a_{\nu}=\prod_{\alpha \in \Phi^-}a_{\alpha}^{\nu_{\alpha}}$. In some arguments we will have to order the exponents $\nu_{\alpha}$'s. We use the following lexicographic order. Let $\alpha=(i,j)$ and $\alpha^{\prime}=(k,l)$. Then $\nu_{\alpha}$ comes before $\nu_{\alpha^{\prime}}$ if and only if $i<k$ or $i=k$ and $j<l$, i.e. $\nu=(\nu_{2,1},\nu_{3,1},\nu_{3,2},...,\nu_{n,n-1})$ (see also the  discussion before equation \ref{eq:fanalytic}). For $N \geq 0$, let $\tau_N$ be the natural truncation
$$\mathcal{A} \rightarrow K[a]_N:=\oplus_{|\nu| \leq N}Ka^{\nu}.$$ The later space is the space of polynomials in several variables with total degree $\leq N$. As $\tau_N$ is equivariant under the action of the diagonal subgroup of $G$ given by formulas \ref{eq:analyticexpression1} and \ref{eq:diagonal1} and the associated characters of the diagonal torus of $G$ are linearly independent, $\tau_N(\mathcal{X})$ is a direct sum of monomials given by $$\tau_N(\mathcal{X}):=\mathcal{X}_N=\{\sum_{\nu \in M_N}c_{\nu}a^{\nu}\}$$ where $M_N$ is the set of exponents of $a$ of elements in $\mathcal{X}_N$. If $N \leq N^{\prime}$ and $\nu \in M_N$, then by surjectivity $$K[a]_{N^{\prime}} \twoheadrightarrow K[a]_N,$$ we obtain $\mu \in M_{N^{\prime}}$. Conversely, $\nu \in M_{\N^{\prime}}$ and $|\nu| \leq N$ implies $\nu \in M_N$. Therefore, the multi-sets $M_N$ and $M_{N^{\prime}}$ are compatible and thus there exists $M$ (the exponents of elements of $\mathcal{X}$) such that 
\begin{enumerate}
\item $f \in \mathcal{X} \implies c_{\nu}=0$ for all $\nu \notin M$.
\item If $\nu \in M$, $a^{\nu} \in \tau_N(\mathcal{X})$ for all $N \geq |\nu|$, thus there exists $$f:=a^{\nu}+\sum_{|r|>N}c_ra^r \in \mathcal{X}$$ (here $r=(r_{\alpha}, \alpha \in \Phi^-) \in \N^d, |c_r| \rightarrow 0$) \label{specialf}.
\end{enumerate}
For $\alpha \in \Phi^-$, let $Y_{\alpha} \in \mathfrak{g}=\text{Lie}(G)$ be the infinitesimal generator associated to the unipotent subgroup $1+yE_{\alpha}, y \in \Z_p$, $E_{\alpha}$ being the standard elementary matrix at $\alpha$. 

\begin{lemma}\label{0inM}
The multi-index $0 \in M.$ 
\end{lemma}
\begin{proof}
$M \neq $ null, because if so, then $\mathcal{X}=0$, which is not true by assumption. Now if $\nu=(\nu_{\alpha},\alpha \in \Phi^-) \in M$, then by (\ref{specialf}) above $$f=a^{\nu}+\sum_{|r|>N}c_{r}a^{r} \in \mathcal{X}$$ (here $N \geq |\nu|, r \in \N^d$).
By \ref{arparina}, the action of $Y_{\beta}=Y_{(i,j)}$ on $f$ (where $\beta=(i,j) \in \Phi^-$ is fixed) is given by 
\begin{align*}
\begin{split}
Y_{\beta}(f)&=\frac{d}{dy}\Big|_{y=0}\Big((\prod_{\underset{u=i \implies v>j}{\alpha=(u,v)}}^{}a_{\alpha}^{\nu_{\alpha}})(\prod_{k=1}^j(a_{i,k}+ya_{j,k})^{\nu_{i,k}})\\
& \qquad \qquad + \sum_{|r|>N}c_r(\prod_{\underset{u=i \implies v>j}{\alpha=(u,v)}}^{}a_{\alpha}^{r_{\alpha}})(\prod_{k=1}^j(a_{i,k}+ya_{j,k})^{r_{i,k}})\Big)\\
& = \prod_{\underset{u=i \implies v>j}{\alpha=(u,v)}}a_{\alpha}^{\nu_{\alpha}}(\sum_{l=1}^j\nu_{i,l}a_{j,l}a_{i,l}^{\nu_{i,l}-1}\prod_{\underset{k \neq l}{k \in [1,j]}}a_{i,k}^{\nu_{i,k}})\\
& \qquad \qquad \sum_{|r|>N}c_r\prod_{\underset{u=i \implies v>j}{\alpha=(u,v)}}a_{\alpha}^{r_{\alpha}}(\sum_{l=1}^jr_{i,l}a_{j,l}a_{i,l}^{r_{i,l}-1}\prod_{\underset{k \neq l}{k \in [1,j]}}a_{i,k}^{r_{i,k}})\\
&=A+\sum_{|r|>N}c_rB
\end{split}
\end{align*}

The first term in the R.H.S of the above equation is 
 $$A:=\prod_{\underset{u=i \implies v>j}{\alpha=(u,v)}}a_{\alpha}^{\nu_{\alpha}}(\sum_{l=1}^j\nu_{i,l}a_{j,l}a_{i,l}^{\nu_{i,l}-1}\prod_{\underset{k \neq l}{k \in [1,j]}}a_{i,k}^{\nu_{i,k}})=\sum_{l=1}^j\nu_{i,l}a_{j,l}^{\nu_{j,l}+1}a_{i,l}^{\nu_{i,l}-1}\prod_{\underset{\alpha \neq (j,l)}{\alpha \neq (i,l)}}a_{\alpha}^{\nu_{\alpha}}$$ and  $$B=\prod_{\underset{u=i \implies v>j}{\alpha=(u,v)}}a_{\alpha}^{r_{\alpha}}(\sum_{l=1}^jr_{i,l}a_{j,l}a_{i,l}^{r_{i,l}-1}\prod_{\underset{k \neq l}{k \in [1,j]}}a_{i,k}^{r_{i,k}})=\sum_{l=1}^jr_{i,l}a_{j,l}^{r_{j,l}+1}a_{i,l}^{r_{i,l}-1}\prod_{\underset{\alpha \neq (j,l)}{\alpha \neq (i,l)}}a_{\alpha}^{r_{\alpha}}$$
 
 Notice that the monomials in $B$ has total degree $|r|$ except the term (when $l=j$) $r_{i,j}a_{i,j}^{r_{i,j}-1}\prod_{\alpha \neq (i,j)}a_{\alpha}^{r_{\alpha}}$ (note that $a_{j,j}=1$ by convention) which has total degree $|r|-1$. 
 
 As $Y_{(i,j)}(f) \in \mathcal{X}$, we see that $(\nu_{\alpha},\nu_{i,j}-1,\alpha \in \Phi^-,\alpha \neq (i,j)) \in M$; these are the exponents when we take $l=j$ in $A$. This shows that if $M \neq $ null, then $ 0 \in M$ because we can descend the $v_{i,j}$'s successively for every negative root  $(i,j)$ and this completes the proof of lemma \ref{0inM}.
\end{proof}

\begin{lemma}\label{1inM}
The constants $a^{0} \in \mathcal{X}$.
\end{lemma}
\begin{proof}
Let $T_k \in \mathfrak{g}$ be the infinitesimal generator associated to the diagonal subgroup $Diag(1,...,t_k,...,1)$, $t_k\in 1+p\Z_p$, $t_k$ is at the $(k,k)$-th place. By lemma \ref{0inM}, $0\in M$. This implies that $$f=c_0+\sum_{|r|>0}c_ra^r \in \mathcal{X}$$ ($c_0 \neq 0$). We will essentially follow the proof given by Clozel for $GL(2)$. By equation \ref{diago}, from the action of $Diag(1,...,t_k,...,1)$ on $f$, the function obtained from $T_k(f),$

%is given by 
%$$Diag(1,...,t_k,...,1)(f)=\Big(c_0+\sum_{|r|>0}c_r(\prod_{\underset{u,v \neq k}{\alpha=(u,v)}}a_{\alpha}^{r_{\alpha}})(\prod_{\underset{i \in [k+1,n]}{\delta=(i,k)}}a_{\delta}^{r_{\delta}}t_k^{r_{\delta}})(\prod_{\underset{j \in [1,k-1]}{\beta=(k,j)}}a_{\beta}^{r_{\beta}}t_k^{-r_{\beta}})\Big)\chi_k(t_k)$$

\begin{equation}
\label{iti}
\sum_{|r|>0}c_r(\sum r_{\delta}-\sum r_{\beta})a^r \in \mathcal{X}
\end{equation}
 where 
$\sum r_{\delta}$ is  $\sum_{\underset{i \in [k+1,n]}{\delta =(i,k)}}r_{\delta}$ and $\sum r_{\beta}$ is  $\sum_{\underset{j \in [1,k-1]}{\beta =(k,j)}}r_{\beta}$.

 The function obtained from $T_k^{p-1}(f),$ 
$$\sum_{|r|>0}c_r(\sum r_{\delta}-\sum r_{\beta})^{p-1}a^r \in \mathcal{X}.$$
This implies that $$E_k f:=c_0+\sum_{|r|>0}c_r(1-(\sum r_{\delta}-\sum r_{\beta})^{p-1})a^r \in \mathcal{X}.$$
If $p \mid \sum r_{\delta}-\sum r_{\beta}$, then $(1-(\sum r_{\delta}-\sum r_{\beta})^{p-1})^l \rightarrow 1$ as $l \rightarrow \infty$. If $p \nmid  \sum r_{\delta}-\sum r_{\beta}$, then $(1-(\sum r_{\delta}-\sum r_{\beta})^{p-1})^l \rightarrow 0$ as $l \rightarrow \infty$. Then $$A_{k,1}f:=c_0+\sum_{\underset{p \mid \sum r_{\delta}-\sum r_{\beta}}{|r|>0}}c_ra^r \in \mathcal{X}.$$ Similar to \ref{iti}, applying now, the transformation $T_k$ on $A_{k,1}f$, dividing by $p$, and iterating all the above steps, we see that
$$A_{k,2}(f):=c_0+\sum_{\underset{p^2 \mid \sum r_{\delta}-\sum r_{\beta}}{|r|>0}}c_ra^r \in \mathcal{X}.$$ Iterating again and again $s$ times, for $s \in \N$, we obtain
$$A_{k,s}(f):=c_0+\sum_{\underset{p^s \mid \sum r_{\delta}-\sum r_{\beta}}{|r|>0}}c_ra^r \in \mathcal{X}.$$ 
This implies, for $s \in \N$, 
\begin{equation}
\label{what}
(\prod_{k=1}^nA_{k,s})(f)=c_0+Q_s(f)\in \mathcal{X}
\end{equation}
where $Q_s(f)=\sum c_ra^r$ where the sum runs over all $r=(r_{\alpha}, \alpha \in \Phi^-)$ with $|r|>0$ such that for all $k \in [1,n]:$
$$p^s \mid \sum_{\substack{\delta=(i,k)\\i\in[k+1,n]}}r_{\delta}-\sum_{\substack{\beta=(k,j)\\j\in [1,k-1]}}r_{\beta}.$$
%where $$Q_s(f)=\sum_{\substack{|r|>0\\p^s \mid \sum r_{\delta}-\sum r_{\beta}\\k \in [1,n]}}c_ra^r.$$
We need to show that $Q_s(f) \rightarrow 0$ as $s \rightarrow \infty$, i.e., we have to show that 
\begin{align}
&\forall N_{\epsilon}, \exists S, \text { such that } \forall s>S, val_p(c_r)>N_{\epsilon},
\forall r \in \N^d \text{ such that } |r|>0,\label{toshow1}\\
& \qquad \qquad \text{ and }p^s \mid \Big(\sum_{\substack{\delta=(i,k)\\i\in[k+1,n]}}r_{\delta}-\sum_{\substack{\beta=(k,j)\\j\in [1,k-1]}}r_{\beta}\Big)  \text{ \hspace{.5cm}      }  \text{ for all } k\in [1,n]\label{toshow2}. 
\end{align}
But as $f$ is globally analytic, $|c_r| \rightarrow 0$ as $|r|=\sum_{\alpha \in \Phi^-}r_{\alpha} \rightarrow \infty$, 
which means that 
\begin{align}
& \forall N_{\epsilon}, \exists S^{\prime} \text{ such that whenever } |r|>S^{\prime} \label{have1}\\
& \text{ we have }val_p(c_r)>N_{\epsilon} \label{have2}.
\end{align}
Any $S$ such that $p^S>S^{\prime}$ will work in \ref{toshow1}. This is because, take any $r$, such that $|r|>0$ and 
 $$p^s \mid \Big(\sum_{\substack{\delta=(i,k)\\i\in[k+1,n]}}r_{\delta}-\sum_{\substack{\beta=(k,j)\\j\in [1,k-1]}}r_{\beta}\Big) \text{ \hspace{.5cm}      } \text{ for all } k\in [1,n]$$

(i.e. satisfying equation \ref{toshow2}) with $s>S$ (cf. (\ref{toshow1})). \\

For $k=1$,  (\ref{toshow2}) implies $p^s \mid r_{2,1}+r_{3,1}+r_{4,1}+\cdots +r_{n,1}$ which means $ r_{2,1}+r_{3,1}+r_{4,1}+\cdots +r_{n,1}\geq p^s>S^{\prime}$ except when $r_{2,1}=r_{3,1}=r_{4,1}=\cdots =r_{n,1}=0$. If this happens, then consider (\ref{toshow2}) with $k=2$, i.e. $p^s \mid r_{3,2}+r_{4,2}+\cdots +r_{n,2}-(r_{2,1}=0),$ i.e. $r_{3,2}+r_{4,2}+\cdots +r_{n,2} \geq p^s >S^{\prime}$ except when $r_{3,2}=r_{4,2}=\cdots =r_{n,2}=0$. Repeating this process, since we have started with an $r$ such that $|r|>0$, we see that any $r$ as in (\ref{toshow2}), with $|r|>0$, satisfies $|r|>S^{\prime}$ for all $s >S$ and this by (\ref{have1}) and (\ref{have2}) implies that $val_p(c_r)>N_{\epsilon}$ which was the desired condition in (\ref{toshow1}). (Here $S$ is chosen such that $p^S>S^{\prime}$). This shows that $Q_s(f) \rightarrow 0$ as $s \rightarrow \infty$ which gives $c_0 \in \mathcal{X}$ (cf. (\ref{what})). This completes the proof of lemma \ref{1inM}.
\end{proof}

In the following, we complete the proof of Theorem \ref{thm:holomorphic} which was to find conditions such that the globally analytic representation $\mathcal{A}$ of $G$ is topologically irreducible. It uses an argument concerning Verma modules and the condition of irreducibility of $\mathcal{A}$ comes from a result of Bernstein-Gelfand determining the condition of irreducibility of that Verma module. 

 With notations as in section \ref{sub:sectionbigger}, let $\mathfrak{g}=\Lie(G), \mathfrak{h}=\Lie (T_0), \mathfrak{b} \text{ } (\text{resp. } \mathfrak{b}^-)$ be the upper (resp. lower) triangular Borel subalgebra containing $\mathfrak{h}$.  Let $\mathfrak{u}^{-}=\Lie(U)$. Recall that here $c_i$'s $\in K$ are such that $\chi_i(t)=t^{c_i}$, for $t \rightarrow 1$. Let $$V_{-\mu}:=U(\mathfrak{g}) \otimes_{U(\mathfrak{b}^-)} K$$ be the Verma module
 where $U(\mathfrak{b}^-)$ acts on $K$ via the action of $\mathfrak{b}^-=\mathfrak{u}^- \oplus \mathfrak{h},\mathfrak{u}^-$ acting trivially and $\mathfrak{h}$ via 
 
 $-\mu \in \mathfrak{h}^{\star}=\Hom(\mathfrak{h},K)$ given by 
 \begin{equation}\label{mu}
 \mu=(-c_1,...,-c_n):Diag(t_1,...,t_n)\mapsto \sum_{i=1}^n-c_it_i
 \end{equation}
where $t=(t_i) \in \mathfrak{h}, U(\mathfrak{g})$ is the Universal enveloping algebra of $\mathfrak{g}$.   (Note that Dixmier has a different normalization for the Verma module     \cite[section 7.1.14]{Dix}).

Let $\mathcal{A}_{\fin}$ be the set of polynomials within the rigid analytic functions $\mathcal{A}$. For $k\in [1,n]$, let $T_k \in \mathfrak{h}$ be the infinitesimal generator associated to the one parameter diagonal subgroup $Diag(1,...,t_k,...,1)$, $t_k\in 1+p\Z_p$, $t_k$ is at the $(k,k)$-th place and $f=a^r \in \mathcal{A}_{\fin}$. The elements $T_k$ form a basis of $\mathfrak{h}$. By equations \ref{eq:diag} and \ref{diago}, the action of $Diag(1,...,t_k,...,1)$ on $f$ is given by 
$$Diag(1,...,t_k,...,1)(f)=\Big((\prod_{\underset{u,v \neq k}{\alpha=(u,v)}}a_{\alpha}^{r_{\alpha}})(\prod_{\underset{i \in [k+1,n]}{\delta=(i,k)}}a_{\delta}^{r_{\delta}}t_k^{r_{\delta}})(\prod_{\underset{j \in [1,k-1]}{\beta=(k,j)}}a_{\beta}^{r_{\beta}}t_k^{-r_{\beta}})\Big)\Big(\chi_k(t_k)\Big)$$
As $\chi_k(t_k)=t_k^{c_k}$, so the action of $T_k$ on $f$ is 

\begin{align}
T_k \cdot f&=c_ka^r+ (\sum_{\underset{i \in [k+1,n]}{\delta=(i,k)}}r_\delta - \sum_{\underset{j \in [1,k-1]}{\beta=(k,j)}}r_\beta)a^r\\
&=(c_k+ \sum_{\underset{i \in [k+1,n]}{\delta=(i,k)}}r_\delta - \sum_{\underset{j \in [1,k-1]}{\beta=(k,j)}}r_\beta)a^r\\
%&= - (T_k \cdot g) \text{  } ( \text{where } g=X^r \otimes 1 = \prod_{i=1}^dX_{\alpha_i}^{r_{\alpha_i}} \otimes 1 \in M(\lambda) \text { and } \text{action as in } \ref{eq:vermaaction}).
&=(-\mu-\sum_{i=1}^d \alpha_ir_{\alpha_i})(T_k)a^r
\end{align}
Here $\alpha_i$'s are the negative roots.\\

Thus if $H \in \mathfrak{h}$, then 
\begin{equation}\label{twojstar}
H\cdot a^r=(-\mu-\sum_{i=1}^d \alpha_ir_{\alpha_i})(H)a^r
\end{equation}
Decomposing $\mathcal{A}_{\fin}=\oplus_{\xi \in \mathfrak{h}^{\star}}\mathcal{A}_{\fin}(\xi)$ in the form of $\mathfrak{h}$- eigenspaces, we see from \ref{twojstar} that the monomials $a^r$ are $\mathfrak{h}$-finite and and the dimensions of eigenspaces of $\mathcal{A}_{\fin}$ under $\mathfrak{h}$ are finite: The eigenvectors are of the form $\xi \in -\mu-\sum_{i=1}^d \N \alpha_i \in \mathfrak{h}^{\star}$ and the multiplicity $\mult(\xi)=\dim \mathcal{A}(\xi)$ of $\xi$ equals 
\begin{equation}\label{mult}
\dim \mathcal{A}(\xi)=\mult(\xi)=\{\text{number of families } (r_{\alpha_i})\in \N^d \mid \xi=-\mu-\sum_{i=1}^dr_{\alpha_i}\alpha_i\},
\end{equation}
which is finite. 

With $f_0=1 \in \mathcal{A}_{\fin}$, $H\cdot f_0=-\mu(H)f_0$ and the action  $\mathfrak{u}^- \cdot f_0=0$ because the action of any element of $\mathfrak{u}^-$ on $f_0$ is given by derivation (cf. proof of  Lemma \ref{0inM}). So, the map $u \rightarrow u\cdot f_0 $  for $u \in \mathfrak{g}$ induces a $\mathfrak{g}$-homomorphism 
\begin{align}\label{VermatoA}
\phi: V_{-\mu} &\rightarrow \mathcal{A},
\end{align}
where $V_{-\mu}:=U(\mathfrak{g}) \otimes_{U(\mathfrak{b}^-)}K.$

Moreover, $v \in V_{-\mu}$ implies $v$ is $\mathfrak{h}$ - finite (cf. \cite[Chapter 7]{Dix}). This gives $\phi(v) \in \mathcal{A}$ is $\mathfrak{h}$ - finite which means that $\phi(v) \in \mathcal{A}_{\fin}$. This is because equation \ref{twojstar} gives, by continuity, that $f \in \mathcal{A}$, $f=\sum_{r=(r_{\alpha_i})}c_ra^r$ implies 
\begin{equation}\label{hotstar}
H\cdot f=\sum_{r=(r_{\alpha_i})}(-\mu-\sum_{i=1}^dr_{\alpha_i}\alpha_i)(H)c_ra^r.
\end{equation}
Then $H\cdot f=\lambda f$ implies $\lambda=(-\mu-\sum_{i=1}^dr_{\alpha_i}\alpha_i)(H)$ if $c_r \neq 0$. Therefore the cardinality of the set $\{c_r \neq 0\}$ is finite, the $\mathfrak{h}$-finite vectors of $\mathcal{A}$ are just $\mathcal{A}_{\fin}.$\\

The map $\phi:V_{-\mu} \rightarrow \mathcal{A}_{\fin}$ in \ref{VermatoA} is clearly non-zero because the vector $1\in V_{-\mu}$ goes to $f_0$.

\begin{lemma}\label{irreA}
If the Verma module $V_{-\mu}$ is irreducible then the globally analytic $G$- representation $\mathcal{A}$ is irreducible.
\end{lemma} 
\begin{proof}
Suppose that the Verma module $V_{-\mu}$ is irreducible. Then the map $\phi:V_{-\mu} \rightarrow \mathcal{A}_{\fin}$ is injective. Also by \ref{mult}, under the action of $\mathfrak{h}$, since the eigenvectors of $V_{-\mu}$ and $\mathcal{A}_{\fin}$ and their multiplicities match, that is $\dim \mathcal{A}(\xi)=\dim \mathcal{A}_{\fin}(\xi)= \dim V_{-\mu}(\xi),$ we deduce that $\phi$ is an isomorphism.

Indeed the dimension of $\dim \mathcal{A}(\xi)$ is given by \ref{mult}, on the other hand, using that our Verma module $V_{-\mu}$ is defined by $\mathfrak{b}^-$ and $-\mu$ (rather than $\lambda-\rho^-$ as in Dixmier's parametrization \cite[7.1.4]{Dix}), Dixmier's formula \cite[7.1.6]{Dix} yields
$$\dim V_{-\mu}(\xi)=\mult(\xi)=\{\text{number of families } (r_{\alpha_i})\in \N^d \mid \xi=\lambda-\rho^--\sum_{i=1}^dr_{\alpha_i}\alpha_i\}$$
where $\rho^-=\frac{1}{2}\sum_{\alpha \in \Phi^-}\alpha$ is half the sum of negative roots (because notice that we have used $\mathfrak{b}^-$ to define the Verma module instead of Dixmier's $\mathfrak{b}^+$). We easily see  that the above dimension  $\dim V_{-\mu}(\xi)$ is equivalent to $\dim \mathcal{A}(\xi)$ (\ref{mult}) with $\lambda-\rho^-=-\mu$.

So $V_{-\mu} \cong \mathcal{A}_{\fin}$. Suppose $\mathcal{X}$ is a nonzero closed subspace of $\mathcal{A}$. Then by lemma \ref{1inM}, we have $1 \in \mathcal{X}$. This gives $\mathcal{A}_{\fin}=U(\mathfrak{g})\cdot 1 \subset \mathcal{X}$. Since $\mathcal{X}$ is closed $\mathcal{X}=\mathcal{A}$.
\end{proof}

Now we prove the converse of Lemma \ref{irreA}. \\

Recall that a closed subspace of $\mathcal{A}$ is $G$-invariant if and only if it is invariant by $\mathfrak{g}$ (\cite[Proposition 2.4]{Clozel1}). Moreover it follows from the definition of globally analytic representations (compare \cite[Section 2.2]{Clozel1}) that the action of $\mathfrak{g}$ on $\mathcal{A}$ is continuous. If $V \subset \mathcal{A}_{\fin}$ is invariant by $\mathfrak{g}$, it follows that its closure $\overline{V}$ is $G$-invariant. 

Recall that $\mathcal{A}_{\fin}$ is the set of $\mathfrak{h}$-finite vectors in $\mathcal{A}$. In particular, if $\mathcal{X} \subset \mathcal{A}$ is closed, the space $\mathcal{X}_{\mathfrak{h}-\fin}$ of $\mathfrak{h}$-finite vectors in $\mathcal{X}$ is $\mathcal{X} \cap \mathcal{A}_{\fin}$.
\begin{lemma}\label{obv}
Assume $V \subset \mathcal{A}_{\fin}$ is invariant by $\mathfrak{g}$. Then $V=\overline{V} \cap \mathcal{A}_{\fin}=\overline{V}_{\mathfrak{h}-\fin}.$
\end{lemma}
\begin{proof}
By \ref{mult}, $\mathcal{A}(\xi)$ is the subspace of the Tate algebra spanned by a finite number of monomials $a^r$. In particular, the obvious projection $p_{\xi}:\mathcal{A} \rightarrow \mathcal{A}(\xi)$ is continuous. Assume $v \in \overline{V} \cap \mathcal{A}_{\fin}$. Thus $v \in \oplus_{\xi}\mathcal{A}(\xi)$ (finite sum of finite-dimensional subspaces) and $v =\lim v_m, v_m \in V$. If $P$ is the projection on $\oplus_{\xi}\mathcal{A}(\xi)$, $v=Pv=\lim Pv_m$. But $Pv^{\prime} \in V \cap \oplus_{\xi}\mathcal{A}(\xi)$ for any $v^{\prime} \in V$. Thus $v \in V$, as a limit in a finite-dimensional space.
\end{proof}
Lemma \ref{obv} obviously gives the following Corollary.
\begin{corollary}\label{impcorirr}
Suppose $V$ is a non-zero proper subspace of $\mathcal{A}_{\fin}$ stable by $\mathfrak{g}$. Then $\overline{V}$ is a non-zero proper closed $G$-invariant subspace of $\mathcal{A}$.
\end{corollary}
\begin{lemma}\label{irreAconv}
If the globally analytic $G$- representation $\mathcal{A}$ is irreducible then  the Verma module $V_{-\mu}$ is irreducible.
\end{lemma}
\begin{proof}

Let $W \subset \mathcal{A}_{\fin}$ be the image of $V_{-\mu}$ by $\phi$. Then $W \neq 0$. If $\mathcal{A}$ is an irreducible $G$-module, $W=\mathcal{A}_{\fin}$ by corollary \ref{impcorirr}. Thus we have a surjective map $\phi: V_{-\mu} \rightarrow \mathcal{A}_{\fin}$. But, as we noticed, the dimensions of $V_{-\mu}(\xi)$ and of $\mathcal{A}_{\fin}(\xi)$ coincide. This implies that $\phi$ is an isomorphism. On the other hand (again by the corollary \ref{impcorirr}) $W$ is irreducible. Thus $V_{-\mu}$ is irreducible.

\end{proof}

\begin{comment}

Suppose that $V_{-\mu}$ is reducible. Recall that we have a non-zero map $\phi:V_{-\mu} \rightarrow \mathcal{A}_{\fin}$. The image $Image(\phi)$ of $\phi$ is  non-zero . 

First suppose that $Image(\phi)$ is a proper subspace of $\mathcal{A}_{\fin}$. Then by lemma \ref{implemmatoshow}, its closure is a proper subspace of $\mathcal{A}$ which gives that $\mathcal{A}$ is reducible. 

Now suppose that $Image(\phi)=\mathcal{A}_{\fin}$. Then, as in the proof of lemma \ref{irreA}, since the eigenvectors and their multiplicities of $V_{-\mu}$ and $\mathcal{A}_{\fin}$ match we deduce that $V_{-\mu} \cong \mathcal{A}_{\fin}$. But since $V_{-\mu}$ is reducible, by the proof of Theorem 7.6.24 of \cite{Dix}, there exists a proper non-zero $\mathfrak{g}$-stable, $\mathfrak{h}$-finite subspace $V^{\prime}$  of the Verma module $V_{-\mu}$  and hence of $\mathcal{A}_{\fin}$. Again by lemma \ref{implemmatoshow}, $\overline{V^{\prime}}$ is a proper non-zero closed subspace of $\mathcal{A}$ which shows that $\mathcal{A}$ is reducible.
\end{comment}

Now we determine the condition when the Verma module $V_{-\mu}$ is irreducible. Recall that $$\mu=(-c_1,...,-c_n):Diag(t_1,...,t_n)\mapsto \sum_{i=1}^n-c_it_i$$
where $t=(t_i) \in \mathfrak{h}$. For negative root $\alpha =(i,j), i>j$, let $H_{\alpha=(i,j)}$ be the matrix $E_{i,i}-E_{j,j}$ where $E_{i,i}$ is the standard elementary matrix.
\begin{lemma}\label{Vermairred}
 The Verma module   $V_{-\mu}$ is   irreducible if and only if for all $\alpha=(i,j) \in \Phi^-$,  $(-\mu)(H_{\alpha=(i,j)}) + i-j \notin \{1,2,3,...\}$.
\end{lemma}
\begin{proof}

Let $\rho^-=\frac{1}{2}\sum_{\alpha \in \Phi^-}\alpha$. For $\alpha=(i,j) \in \Phi^-$, $H_{\alpha}=H_{i+1,i}+\cdots +H_{j,j-1}$ and $\rho^-(H_{k+1,k})=1$. This gives that $\rho^-(H_{\alpha=(i,j)})=i-j$. By Theorem 7.6.24 of \cite{Dix}, the condition of irreducibility of our $V_{-\mu}$ is $(-\mu+\rho^-)(H_{\alpha}) \notin \{1,2,3,...\}$ for all negative roots $\alpha \in \Phi^-$ (This is because Dixmier's $\mathfrak{b}^+$ is our $\mathfrak{b}^-$ and so we have to work with negative roots.) This gives the condition $(-\mu)(H_{\alpha}) + i-j \notin \{1,2,3,...\}$ 
\end{proof}
Lemma \ref{Vermairred}, Lemma \ref{irreAconv}, and Lemma \ref{irreA} together proves Theorem \ref{thm:holomorphic}. [Q.E.D]

\subsection{~}\label{sub:lastsectionbasechange}
With $L$ an unramified finite extension of $\mathbb{Q}_p$. All the arguments of section \ref{sub:sectionbigger} extend automatically to the group $G(L)$. As $L$ is unramified, the conditions on the character $\chi$ to be analytic, that is those given by \ref{eq:unramifiedK}, remain unchanged. Moreover, note that the representation $\mathcal{A}(B_1^{\frac{n(n-1)}{2}},K)$ (where now $B_1^{\frac{n(n-1)}{2}}$ is seen as a product of $\frac{n(n-1)}{2}$ closed rigid balls of radius $1$ as an $L$-analytic space) given by the lemmas \ref{lem:lemmadiagonal1}, \ref{lem:lemmalower1},  \ref{lem:lemmaupper2} are $L$-analytic. The restriction of $\mathcal{A}(B_1^{\frac{n(n-1)}{2}},K)$ to $G(\mathbb{Q}_p)$ is simply the previous representation. Indeed, the representation of $G(L)$ is obtained from the representation of $G(\mathbb{Q}_p)$ by holomorphic base change (cf. section \ref{sub:holomorphicbasechange} and \cite[prop. $3.1$]{Clozel1}). Denote by $I_{\mathbb{Q}_p}(\chi)$ and $I_L(\chi)$, respectively, the two \textit{globally analytic} representations (the character $\chi$ is defined by the parameters $(c_1,...,c_n)$, we agree to identify the characters for the two fields). Then we have:
\begin{theorem}\label{thm:holomorphic2}
For a given embedding $L \hookrightarrow K$, with $\mu $ as in \ref{mu}, if  $-\mu(H_{\alpha}) + i-j \notin \{1,2,3,...\}$ for all $\alpha=(i,j) \in \Phi^-$, then $I_L(\chi)$ is an admissible, irreducible (under  both $G(L)$ and $G(\mathbb{Q}_p)$) globally analytic representation and it is the holomorphic base change of $I_{\mathbb{Q}_{p}}(\chi)$.
\end{theorem}
$I_L(\chi)$ is admissible, as holomorphic base change respects admissibility \cite[prop. $3.1$]{Clozel1}.
With the notations of section \ref{sub:holomorphicbasechange}, define the full (Langlands) base change of $I_{\mathbb{Q}_p}$ to be the representation of $\Res_{L/\mathbb{Q}_p}G(\mathbb{Q}_p)$ on $\widehat{\otimes}_{\sigma} \text{ } (I_L(\chi))^{\sigma}:=I(\chi \circ N_{L/\mathbb{Q}_p})$, where $N_{L/\mathbb{Q}_p}$ is the norm map from $L$ to $\mathbb{Q}_p$ and $\widehat{\otimes}$ is the completed tensor product (see also \cite[def. $3.8$]{Clozel1}) and $\sigma \in \Gal(L/\Q_p)$. Note that, for each factor, the embedding $i:L \rightarrow K$ must be replaced by $i\circ \sigma$.  Finally, we then have
\begin{theorem}\label{thmlastiwahoribase}
Let $\mu $ be as in \ref{mu}, Assume  $-\mu(H_{\alpha}) + i-j \notin \{1,2,3,...\}$ for all $\alpha=(i,j) \in \Phi^-$. Then the completed tensor product $\widehat{\otimes}_{\sigma} \text{ } (I_L(\chi))^{\sigma}$ is irreducible, and is the representation of $G(L)$ on the space of globally analytic vectors, induced from $\chi \circ N_{L/\mathbb{Q}_p}$.
\end{theorem}
\begin{proof}
Notice that by assumption, each factor in the completed tensor product is irreducible and admits the same description as in theorem \ref{thm:holomorphic2}.  The space of the representation $I(\chi \circ N_{L/\mathbb{Q}_p})$ is $\widehat{\otimes}_{\sigma} \text{ } \mathcal{A}(U,K)=\mathcal{A}(\Res_{L/\mathbb{Q}_p}U,K)$ which is a space of globally analytic vectors (by theorem \ref{thm:holomorphic2})  in the locally analytic representation $I_{\loc}(\chi \circ N_{L/\mathbb{Q}_p})$ of $\Res_{L/\mathbb{Q}_p}(G)$. The proof of irreducibility of $\widehat{\otimes}_{\sigma} \text{ } (I_L(\chi))^{\sigma}$ follows from Theorem \ref{thm:holomorphic} using a natural generalization of Clozel's argument in Theorem $3.11$, part $(ii)$ of \cite{Clozel1},  revised version. 
\end{proof}

\appendix
\section{Analyticity for the induction from the Weyl orbits of the upper triangular Borel subgroup of B}\label{appendixA}
In this appendix we treat the global analyticity of the principal series induced from Weyl orbits of the Borel subgroup (theorem \ref{thminserted1}). Then we base change our globally analytic representation to $L$ (section \ref{sub:weylbasechange}).
\subsection{}\label{sub:weylanalyticity}
Denote by $\mathbb{P}$ the Borel subgroup of the upper triangular matrices in $GL_n(\mathbb{Q}_p)$,  $\mathbb{T}$ the maximal torus of ${GL_n(\mathbb{Q}_p)}$, $P^+$ the Borel subgroup of the upper triangular matrices in $GL_n(\mathbb{Z}_p)$, $W$  the ordinary Weyl group of $GL_n(\mathbb{Q}_p)$ with respect to $\mathbb{T}$ which is isomorphic to the group of $n \times n$ permutation matrices, ${P_w^+=B \cap wP^+w^{-1}}$, where $B$ is the Iwahori subgroup in section \ref{sub:sectionbigger}, $\ind_{\mathbb{P}}^{GL_n(\mathbb{Q}_p)}(\chi)_{\loc}$ the locally analytic induction, that is:
\[\ind_{\mathbb{P}}^{GL_n(\mathbb{Q}_p)}(\chi)_{\loc}=\{f \in \mathcal{A}_{\loc}({GL_n(\mathbb{Q}_p)},K):f(gb)=\chi(b^{-1})f(g), g \in {GL_n(\mathbb{Q}_p)}, b \in \mathbb{P}\}.\]
The Iwasawa decomposition \cite[sec. 3.2.2]{Orlik1} gives 
\begin{equation}\label{eq:decompositioniwasawa}
\ind_{\mathbb{P}}^{GL_n(\mathbb{Q}_p)}(\chi)_{\loc}\cong \ind_{P^+}^{GL_n(\mathbb{Z}_p)}(\chi)_{\loc}
\end{equation}
as $GL_n(\mathbb{Z}_p)$-equivariant topological isomorphism.
By the Bruhat-Tits decomposition (\textit{loc. cit.} and \cite[sec. 3.5]{Cartier1}) \[GL_n(\mathbb{Z}_p)=\sqcup_{w \in W}BwP^+,\]
we obtain the decomposition
\begin{equation}\label{eq:bruhattitsdecomposition}
\ind_{P^+}^{GL_n(\mathbb{Z}_p)}(\chi)_{\loc} \cong \oplus_{w \in W} \ind_{P_w^+}^B(\chi^w)_{\loc},
\end{equation}
a $B$-equivariant decomposition of topological vector spaces, where the action of $\chi^w$ is given by $\chi^w(h)=\chi(w^{-1}hw)$. 
Let $\ind_{P_w^+}^B(\chi^w)$ be the space of globally analytic functions of $\ind_{P_w^+}^B(\chi^w)_{\loc}$.
Our goal is to show that for all $w \in W$, $\ind_{P_w^+}^B(\chi^w)$ is a globally analytic representation of $G$. We have already showed, in section \ref{sec:twoiwahori}, that for $w=Id$, $\chi$ analytic, the induction $\ind_{P_0}^B(\chi)$ is a globally analytic representation of $G$. (Note that $B \cap P^+=P_0$). Recall that $U$ is the lower triangular unipotent subgroup of $GL_n(\mathbb{Z}_p)$. 
Consider the decomposition (cf. lemma $3.3.2$ of \cite{Orlik1})
\[B=(wUw^{-1} \cap B)(wP^+w^{-1}\cap B)=(wUw^{-1} \cap B)(P_w^+).\]
For $GL_3$, and $w= \left( {\begin{array}{ccc}
   0 & 0 & 1 \\
   1 & 0 & 0\\
   0 & 1 & 0 \\
  \end{array} } \right)$
  the above decomposition is  like
\[
  B=
  \left( {\begin{array}{ccc}
   \mathbb{Z}_p^{\times} & p\mathbb{Z}_p & p\mathbb{Z}_p \\
   \mathbb{Z}_p & \mathbb{Z}_p^{\times} & p\mathbb{Z}_p\\
   \mathbb{Z}_p & \mathbb{Z}_p & \mathbb{Z}_p^{\times} \\
  \end{array} } \right)
  =\left( {\begin{array}{ccc}
     1 & p\mathbb{Z}_p & p\mathbb{Z}_p \\
     0 & 1 & 0\\
     0 & \mathbb{Z}_p & 1 \\
    \end{array} } \right)
    \left( {\begin{array}{ccc}
       \mathbb{Z}_p^{\times} & 0 & 0 \\
       \mathbb{Z}_p & \mathbb{Z}_p^{\times} & p\mathbb{Z}_p\\
       \mathbb{Z}_p & 0 & \mathbb{Z}_p^{\times} \\
      \end{array} } \right).
\]
 For a character $\chi$ of $\mathbb{T}\cap GL_n(\mathbb{Z}_p)$, we extend it to a character of $P_w^+$ by acting trivially on the non-diagonal elements of $P_w^+$. By definition,
\[\ind_{P_w^+}^B(\chi)_{\loc}=\{f \in \mathcal{A}_{\loc}(B,K):f(gb)=\chi(b^{-1})f(g), b \in P_w^+, g \in B\}.\]
With the decomposition $B=(wUw^{-1} \cap B)(P_w^+)$, the vector space of locally analytic functions $\ind_{P_w^+}^B(\chi)_{\loc}$ is the same as $\mathcal{A}_{\loc}(wUw^{-1} \cap B,K)$. Let $\mathcal{A}(wUw^{-1} \cap B,K)$ be the subspace of globally analytic functions of $\mathcal{A}_{\loc}(wUw^{-1} \cap B,K)$.
With $i \neq j$ fixed,  $y \in \mathbb{Z}_p$ if $i>j$ and $y \in p\mathbb{Z}_p$ if $i<j$, recall that the action of the one-parameter subgroup on $f \in \mathcal{A}(wUw^{-1} \cap B,K)$ is given by
\begin{align}
(1+yE_{i,j})f(C)&=f((1+yE_{i,j})^{-1}C) \qquad \qquad (\text{with } C\in wUw^{-1} \cap B)\\
&=f((1-yE_{i,j})C)\\
&=f((1-yE_{i,j})wAw^{-1}) \qquad (\text{with } C=wAw^{-1} \text{ for  } A \in U).\label{eq:actionweyl1}
\end{align}

Our goal is to show that this action is globally analytic. 

Since $w^{-1}\in W$, write $w^{-1}$ in the form of a permutation matrix, i.e. ${w^{-1}=\sum_{r=1}^nE_{r,j_r}}$ with $j_r \neq j_s$ for $r \neq s$. Then, \[w^{-1}(1-yE_{i,j})=(\sum_{r=1}^nE_{r,j_r})(1-yE_{i,j})=(\sum_{r=1}^nE_{r,j_r})-yE_{k,j}\]
where $k$ is such that $j_k=i$.
As the inverse of a permutation matrix is its transpose, we obtain 
\begin{align*}
w^{-1}(1-yE_{i,j})w&=\Big((\sum_{r=1}^nE_{r,j_r})-yE_{k,j}\Big)(\sum_{s=1}^nE_{j_s,s})\\
&=1-yE_{k,l}  \qquad (\text {use } j_r \neq j_s \text{ for } r \neq s)
\end{align*} 
where $l$ is such that $j_l=j$.
So we have deduced that 
\begin{equation}\label{eq:thisequation}
(1-yE_{i,j})w=w(1-yE_{k,l}) \qquad (k,l \text{ such that } j_k=i,j_l=j).
\end{equation}
Inserting equation \ref{eq:thisequation} in \ref{eq:actionweyl1} we obtain 
\begin{align}\label{eq:thatequation}
(1+yE_{i,j})f(C)=f(w(1-yE_{k,l})Aw^{-1})
\end{align}
Now, the globally analytic function $f$ on $w(1-yE_{k,l})Aw^{-1}$ equals to some globally analytic function $g$ on $(1-yE_{k,l})A$, because the conjugacy action of $w$ on the matrix $(1-yE_{k,l})A$ is just permuting the entries of $(1-yE_{k,l})A$. So, equation \ref{eq:thatequation} is 
\begin{align*}
f(w(1-yE_{k,l})Aw^{-1})&=g((1-yE_{k,l})A)\\
&=(1+yE_{k,l})g(A) \qquad (\text{recall } A \in U)
\end{align*}
and we know from lemmas \ref{lem:lemmalower1} and \ref{lem:lemmaupper2} that the action of $(1+yE_{k,l})$ on $g(A)$  is globally analytic. Thus, we have shown that 
\begin{lemma}\label{lem:actionweyllowerupper}
The action of the lower and the upper unipotent one-parameter subgroups of $G$ of the form $(1+yE_{i,j})$ on $f \in \mathcal{A}(wUw^{-1} \cap B,K)$ is a globally analytic action.
\end{lemma}
Similar argument also shows that the action of the diagonal subgroup of $G$ on $\mathcal{A}(wUw^{-1} \cap B,K)$ is globally analytic. More precisely, we write $w^{-1}Diag(t_1,...,t_n)w=Diag(t_1^{\prime},...,t_n^{\prime})$ with $(t_1^{\prime},...,t_n^{\prime})$ a permutation of $(t_1,...,t_n)$. Then, with ${C\in wUw^{-1}\cap B}$,
\begin{align*}
Diag(t_1^{-1},...,t_n^{-1})f(C)&=f\Big(Diag(t_1,...,t_n)wAw^{-1}\Big) \qquad (C=wAw^{-1})\\
&=f\Big(w[Diag(t_1^{\prime},...,t_n^{\prime})]Aw^{-1}\Big)\\
&=g\Big(Diag(t_1^{\prime},...,t_n^{\prime})A\Big) \qquad (\text{for some analytic } g)\\
&=Diag(t_1^{-1},...,t_n^{-1})g(A)
\end{align*}
and by lemmas \ref{lem:lemmadiagonal1} and \ref{lem:sufficeaction}, the action of the diagonal subgroup of $G$ on $g(A)$ is a globally analytic action. Therefore, we have shown
\begin{lemma}\label{lem:analyticweyldiagonalaction}
The action of the diagonal subgroup of $G$ on $\mathcal{A}(wUw^{-1} \cap B,K)$ is globally analytic.
\end{lemma}
Recall that the vector space $\mathcal{A}(wUw^{-1} \cap B,K)$ is isomorphic to $\ind_{P_w^+}^B(\chi^w)$. Thus, lemmas \ref{lem:actionweyllowerupper} and \ref{lem:analyticweyldiagonalaction}
together gives
\begin{theorem}\label{thminserted1}
Assume $p>n+1$. Then, for all $w \in W$, the action of the pro-$p$ Iwahori group $G$ on $\ind_{P_w^+}^B(\chi^w)$ is globally analytic.
\end{theorem}

\subsection{}\label{sub:weylbasechange}
Following the notations of section \ref{sub:lastsectionbasechange}, we fix $L$ a finite unramified extension of $\mathbb{Q}_p$ inside $K$. For each $w \in W, $ consider the globally analytic admissible representation $I_{w, \mathbb{Q}_p}(\chi):=\mathcal{A}(wUw^{-1} \cap B,K)$ of $G(\mathbb{Q}_p)$. By section \ref{sub:holomorphicbasechange}, $\mathcal{A}(wUw^{-1} \cap B,K)$ extends naturally to a globally analytic admissible representation of $G(L)$ called the "holomorphic base change" which we denote by $I_{w, L}(\chi)$.  With the notations of section \ref{sub:holomorphicbasechange}, define the full Langland's base change to be the representation of $\Res_{L/\mathbb{Q}_p}G(\mathbb{Q}_p)$ on $\oplus_{w \in W}(\widehat{\otimes}_{\sigma} \text{ } I_{w, L}(\chi)^{\sigma})$ (cf. \cite[sec. 3.5]{Clozel1}). Finally, like theorem \ref{thmlastiwahoribase}, we will then have
\begin{theorem}\label{thm:fullLanglandsweyl}
The Langlands base change   $\oplus_{w \in W}(\widehat{\otimes}_{\sigma} \text{ } I_{w, L}(\chi^w)^{\sigma})$ is a globally analytic admissible representation of $G(L)$. 
\end{theorem}

In conclusion, for $p>n+1$, we have shown that for all $w \in W$,  $\ind_{P_w^+}^B(\chi^w)$ is a globally analytic representation of the pro-$p$ Iwahori $G$ under the analyticity assumption on the character $\chi$. Furthermore we have treated the case of irreducibility of the principal series  when $w=Id$. We hope that it is possible to adapt and generalize  the argument of our irreducibility proof to treat the case when $w \neq Id$. Also it is an interesting future project to determine the globally analytic vectors of more general $p$-adic representations of $GL(2,\Q_p)$, for example the "trianguline"  representation of Colmez \cite{Colmeztri} (see also \cite{ColmezLoc}),  which corresponds to a quotient of principal series. Also one can explore the connection with the globally analytic vectors of $p$-adic representations (under the pro-$p$ Iwahori or a suitable rigid-analytic subgroup of $GL(2)$) and $(\varphi,\Gamma)$-modules \cite{Colmezphi}, similar to the  existing correspondence for locally analytic representations \cite[Sec VI.3]{ColDos}.


\bibliographystyle{amsplain}
\bibliography{base_change_Iwahori,Iwawasa_algebra_and_cornut}

\end{document}